\documentclass[11pt,letterpaper]{amsart}
\usepackage{graphicx, color, url}
\usepackage{mathtools}
\usepackage{amssymb}
\usepackage{enumerate}
\usepackage[margin=1in]{geometry}
\usepackage{pgfplots}
\usepgfplotslibrary{fillbetween}
\usetikzlibrary{calc}
\pgfplotsset{compat=1.12}
\usepackage[final]{pdfpages}

\newtheorem{theorem}{Theorem}[section]
\newtheorem{proposition}[theorem]{Proposition}
\newtheorem{lemma}[theorem]{Lemma}
\newtheorem{corollary}[theorem]{Corollary}
\theoremstyle{definition}

\theoremstyle{remark}

\makeatletter
\newenvironment{proofof}[1]{\par
  \pushQED{\qed}%
  \normalfont \topsep6\p@\@plus6\p@\relax
  \trivlist
  \item[\hskip\labelsep
        \bfseries
    Proof of #1\@addpunct{.}]\ignorespaces
}{%
  \popQED\endtrivlist\@endpefalse
}
\makeatother

\let\Re\undefined
\usepackage{amsopn}

\DeclareMathOperator{\Sub}{Sub}
\DeclareMathOperator{\sing}{sing}
\DeclareMathOperator{\Re}{Re}
\DeclareMathOperator{\diag}{diag}
\DeclareMathOperator{\Gr}{Gr}
\DeclareMathOperator{\Ch}{Ch}
\DeclareMathOperator{\rank}{rank}
\DeclareMathOperator{\mrank}{\mu rank}
\DeclareMathOperator{\argmin}{argmin}
\DeclareMathOperator{\brank}{\overline{\rank}}
\DeclareMathOperator{\spn}{span}

\newcommand{\tp }{{\scriptscriptstyle\mathsf{T}}}
\newcommand{\KK}{\mathbb{K}}
\newcommand{\RR}{\mathbb{R}}
\newcommand{\CC}{\mathbb{C}}
\newcommand{\PP}{\mathbb{P}}

\newcommand{\blue}[1]{#1}

\begin{document}
\title{Complex best $r$-term approximations almost always exist in finite dimensions}
\author{Yang~Qi}
\address{Computational and Applied Mathematics Initiative, Department of Statistics, University of Chicago, Chicago, IL, 60637-1514.}
\email{yangqi@galton.uchicago.edu}
\author{Mateusz Micha\L ek}
\address{Max Planck Institute for Mathematics in the Sciences, 04103 Leipzig, Germany and Polish Academy of Sciences, 00656 Warsaw, Poland.}
\email{Mateusz.Michalek@mis.mpg.de}
\author{Lek-Heng~Lim}
\address{Computational and Applied Mathematics Initiative, Department of Statistics,
University of Chicago, Chicago, IL 60637-1514.}
\email[corresponding author]{lekheng@galton.uchicago.edu}

\begin{abstract}
We show that in finite-dimensional nonlinear approximations, the best $r$-term approximant of a function $f$  almost always exists over $\mathbb{C}$ but that the same is not true over $\mathbb{R}$, i.e., the infimum $\inf_{f_1,\dots,f_r \in Y} \lVert f - f_1 - \dots - f_r \rVert$ is almost always attainable by complex-valued functions $f_1,\dots, f_r$ in $Y$, a set of functions that have some desired structures. Our result extends to functions that possess special properties like symmetry or skew-symmetry under permutations of arguments. For the case where $Y$ is the set of separable functions, the problem becomes that of best rank-$r$ tensor approximations. We show that over $\mathbb{C}$, any tensor almost always has a unique best rank-$r$ approximation. This extends to other notions of tensor ranks such as symmetric rank and alternating rank, to best $r$-block-terms approximations, \blue{and to best approximations by tensor networks.} When applied to sparse-plus-low-rank approximations, we obtain that for any given $r$ and $k$, a general tensor has a unique best approximation by a sum of a rank-$r$ tensor and a $k$-sparse tensor with a fixed sparsity pattern; this arises in, for example, estimation of covariance matrices of a Gaussian hidden variable model with $k$ observed variables conditionally independent given $r$ hidden variables. \blue{The existential (but not the uniqueness) part of our result also applies to best approximations by a sum of a rank-$r$ tensor and a $k$-sparse tensor with no fixed sparsity pattern, as well as to tensor completion problems.}
\end{abstract}

\maketitle

\section{Introduction}\label{sec:intro}

There are numerous problems in scientific and engineering applications that may ultimately be put in the following  form: Given a real or complex-valued function $f : \Omega \to \KK$ (with $\KK = \mathbb{R}$ or $\mathbb{C}$ respectively), find a \emph{best approximation} of $f $ by a sum of functions $f_1,\dots, f_r$ with some special structure, i.e., 
\begin{equation}\label{eq:main}
\min_{f_1,\dots,f_r \in Y} \lVert f - (f_1 + \dots + f_r) \rVert,
\end{equation}
where $Y$ is a subset of functions possessing that special structure. The problem, called the \emph{best $r$-term approximation problem} or \emph{best rank-$r$ approximation problem} is ubiquitous by its generality and simplicity. A slight generalization involves
\begin{equation}\label{eq:join}
\min_{f_1,\dots,f_r \in Y_1; \; g_1,\dots, g_s \in Y_2} \lVert f - (f_1 + \dots + f_r + g_1 + \dots + g_s) \rVert,
\end{equation}
where $Y_1$ and $Y_2$ denote two subsets of functions each with a different structure (e.g., this arises in so-called sparse plus low-rank approximations). More generally, the approximation could involve even more subsets $Y_1,\dots,Y_k$ of functions, each capturing a different structure in the target function $f$.

In practice, if the problem is not already discrete, then it has to be discretized  for the purpose of computations. This usually involves discretizing $\Omega$ (e.g., sampling points, triangulation, mesh generation) or finding a finite-dimensional approximation of the function spaces (e.g., Galerkin method, quadrature) or both (e.g., collocation methods). The result of which is that we may in effect assume that $\Omega$ is a finite set or that the function space $L^2_\KK(\Omega)$ is a finite-dimensional vector space (as our notation implies, we assume that  the norm in \eqref{eq:main} is an $l^2$-norm). In which case, up to a choice of basis, $L^2_\KK(\Omega) \cong \KK^{n}$, where $n  = \dim L^2_\KK(\Omega)$ (or, if $\Omega$ is finite,  $n =  \# \Omega$).

The main issue with \eqref{eq:main} is that in many common scenarios, the approximation problem in \eqref{eq:main} does not have a solution: Let $f \in L^2_\KK(\Omega)$ and $Y \subseteq L^2_\KK(\Omega)$ be a closed subset (under the metric topology induced by $\lVert\,\cdot \,\rVert$), the infimum 
\[
\inf_{f_1,\dots,f_r \in Y} \lVert f - (f_1 + \dots + f_r) \rVert
\] 
is often not attainable when $r > 1$. The reason being that the set of $r$-term approximants
\[
\Sigma^{\circ}_r(Y) \coloneqq \{f \in L^2_\KK(\Omega) \mid f = f_1 + \dots + f_r \text{ for some } f_1, \dots, f_r \in Y\}
\]
is often not closed when $r > 1$.  Let $\Sigma_r(Y)$ denote the closure of $\Sigma^{\circ}_r(Y)$ under the metric topology. Each element in $\Sigma_r(Y)$  is a limit of a sequence of $r$-term approximants but may not itself be an $r$-term approximant. While
\[
\varphi_*  \in \argmin_{\varphi \in \Sigma_r(Y)} \|f - \varphi\|
\]
always exist, it would not in general be an $r$-term approximant, i.e., $\varphi_*$  may fail to be of the required form $f_1 + \dots + f_r$ with $f_i \in Y$ --- this happens when $\varphi_* \in  
\Sigma_2(Y) \setminus \Sigma_2^{\circ}(Y)$. The main result of this article is to  show that it makes a vast difference whether $\KK = \RR$ or $\KK = \CC$ --- for the latter, this failure almost never happens under some mild conditions.

We will elaborate this phenomenon more concretely by studying the case of separable approximations. In this case, $\Omega = \Omega_1 \times \dots \times \Omega_d$ and 
\begin{equation}\label{eq:separable}
Y = \{ \varphi \in L^2_\KK(\Omega) \mid \varphi = \varphi_1 \otimes \dots \otimes \varphi_d \; \text{where} \; \varphi_j \in L^2_\KK(\Omega_j) \}
\end{equation}
is the set of separable functions, i.e., functions $\varphi : \Omega_1 \times \dots \times \Omega_d \to \KK$ of the form
\[
\varphi(x_1,x_2,\dots, x_d) = \varphi_1(x_1) \varphi_2(x_2) \cdots \varphi_d(x_d), \qquad x_j \in \Omega_j, \; j=1,\dots, d.
\]
Take $d = 3$ for simplicity. When $\KK = \RR$, it has been shown \cite[Theorem~8.4]{deSilvaLim08}  that there is a nonempty open subset $\mathcal{O} \subseteq L^2_\RR(\Omega)$ such that for any $f \in \mathcal{O}$, the infimum
\begin{equation}\label{eq:rank2}
\inf_{\varphi_i , \psi_i \in L^2_\RR(\Omega_i)} \|f - \varphi_1 \otimes \varphi_2 \otimes \varphi_3 - \psi_1 \otimes \psi_2 \otimes \psi_3\|
\end{equation}
fails to be attainable. In particular the set of such failures, given that it contains an open set, must have positive volume, i.e., such failures occur with positive probability and cannot be ignored in practice. 
Given any $f \in L^2_\KK(\Omega)$, the \emph{rank} of $f$ is the  integer $r$ such that $f$ is a sum of $r$ separable functions, i.e.,
\begin{equation}\label{def:rank}
\rank(f)  = \min \biggl\{ r \biggm| \sum_{i = 1}^r \varphi_{1,i} \otimes \dots \otimes \varphi_{d,i} \biggr\},
\end{equation}
and so  $\Sigma^{\circ}_r(Y) = \{f \in L^2_\KK(\Omega) \mid \rank(f) \le r\}$.
Therefore the preceding discussion says that the set of $f$ that fails to have a best rank-two approximation has positive volume. In fact, there are more extreme examples \cite[Theorem~8.1]{deSilvaLim08} where \emph{every} $f \in L^2_\RR(\Omega)$ of rank $> 2$ fails to have a best rank-two approximation. Our main result in this article will show that such failures are rare, in fact, they almost never happen, when $\KK = \CC$. 

\blue{We will briefly review the example in \cite[Theorem~8.1]{deSilvaLim08} to give readers a better idea.
Let $\Omega_1 = \Omega_2 = \Omega_3 = \{0,1\}$. One can show that  real-valued functions $f \in L^2_\mathbb{R}(\Omega_1 \times \Omega_2 \times \Omega_3) =L^2_\mathbb{R}(\{0,1\}^3)$ with $\rank_\mathbb{R}(f) > 2$ must take either one of the two following forms:
\begin{align}
f &= f_1 \otimes f_0 \otimes f_0 + f_0 \otimes f_1 \otimes f_0 + f_0 \otimes f_0 \otimes f_1, \label{eq:rank2a}\\
f &= (f_0 + f_1) \otimes f_1 \otimes f_1 + (f_0 - f_1) \otimes f_0 \otimes f_0 + f_1 \otimes (f_0 + f_1) \otimes (f_0 - f_1),\label{eq:rank2b}
\end{align}
where $f_0, f_1 \in L^2_\mathbb{R}(\{0,1\})$ are linearly independent. Let $\mathcal{D}$ and $\mathcal{O}$ be the subsets of functions in $L^2_\mathbb{R}(\{0,1\}^3)$ that take the forms in \eqref{eq:rank2a} and \eqref{eq:rank2b} respectively. One can show that the infimum in  \eqref{eq:rank2}  cannot be attained for any $f \in \mathcal{D} \cup \mathcal{O}$. In addition $\mathcal{O}$ is an open subset in $L^2_\mathbb{R}(\{0,1\}^3)$ with positive volume. So the best two-term approximation problem fails to have a solution on a positive-volumed set.
On the other hand, if one had worked over $\mathbb{C}$, then one can show that the only complex-valued functions $f \in L^2_\mathbb{C}(\Omega_1 \times \Omega_2 \times \Omega_3) =L^2_\mathbb{C}(\{0,1\}^3)$ with $\rank_\mathbb{C}(f) > 2$ are those taking the form in \eqref{eq:rank2a}
where $f_0, f_1 \in L^2_\mathbb{C}(\{0,1\})$ are linearly independent. While the infimum in  \eqref{eq:rank2}  also cannot be attained for such an $f $, the subset of all such functions has zero volume in $L^2_\mathbb{C}(\{0,1\}^3)$. So  the best two-term approximation problem fails to have a solution only on a zero-volumed set.}

Up until this point, we have presented our discussions entirely in the language of function approximation in order to put the problems \eqref{eq:main} and \eqref{eq:join} within the context they most frequently arise, which traditionally goes under the heading of \emph{nonlinear approximation}. Nevertheless, the subject has witnessed many recent breakthroughs and is now studied under a number of different names \cite{Can, CT, CDD, Cyb, Don, DE, GMS, GN, Gross, Tem}. In this article we will undertake  an approach via complex algebraic geometry and real analytic geometry. From our perspective, the problems in \eqref{eq:main} and \eqref{eq:join}  are respectively about \emph{secant varieties} and \emph{join varieties}. We will henceforth drop the function approximation description in order to focus on the crux of the issue.

In the rest of this article, we will let $U$ denote a real vector space and $V$ denote a complex vector space.
In geometric terms,  let $X$ be a closed semianalytic subset of an $\RR$-vector space $U$, and $Z \subseteq X$ be a semianalytic subset of $X$ with $\dim Z < \dim X$. The set $Z$ here represents the `bad points' to be avoided --- given $p \in U$,  we seek a best approximation $x_* \in X$ that attains $\min_{x \in X} \lVert p - x\rVert$ but we also want $x_* \notin Z$.
The example we discussed above has $X = \Sigma_2(Y)$, $Z = \Sigma_2(Y) \setminus \Sigma_2^{\circ}(Y)$, and $f \in L^2_\KK(\Omega)$ in the role of $p \in V$, with $Y$ the set of separable functions as defined in \eqref{eq:separable}. 
By this example, we see that  there can be a nonempty open subset $\mathcal{O} \subseteq U$ such that for $p \in \mathcal{O}$, any best approximation of $p$ lies in $Z$, which in this example represents the functions in $X$ that cannot be written as a sum of two separable functions.

We shall show that this does not happen when $\KK =\CC$: If $X$ is a closed irreducible complex analytic variety in a $\CC$-vector space $V$, and $Z$ is a complex analytic subvariety of $X$ with $\dim Z < \dim X < \dim V$, then any general $p \in V$ will have its best approximation in $X$ but not in $Z$, and moreover this best approximation will be unique. We will apply this to obtain a number of existence and uniqueness results for various nonlinear approximations common in applications.

\section{Subanalytic geometry}

We will study our approximation problems using selected tools from subanalytic geometry  \cite{BierstoneMilman88}, which we review in the following. A \emph{semianalytic} subset of $\RR^n$ is a set that is locally of the form
\[
X = \bigcup_{i = 1}^m \bigcap_{j = 1}^k X_{ij},
\]
where each $X_{ij}$ takes the form 
\begin{gather*}
\{(x_1, \dots, x_n) \in \RR^n \mid f_{ij} (x_1, \dots, x_n) = 0\}
\shortintertext{or}
\{(x_1, \dots, x_n) \in \RR^n \mid f_{ij}(x_1, \dots, x_n) > 0\}
\end{gather*}
for some real analytic function $f_{ij}$. Let $W$ be a linear subspace of $\RR^n$, and $\pi \colon \RR^n \to W$ be a linear projection. A subset $X \subseteq W$ is called \emph{subanalytic} if $X$ is locally of the form $\pi(Y)$ for some relatively compact semianalytic subset $Y \subseteq \RR^n$. Given subanalytic subsets $X \subseteq \RR^m$ and $Y \subseteq \RR^n$, a continuous map $f \colon X \to Y$ is called \emph{subanalytic} if the graph of $f$, \blue{i.e., 
\[
\{(x, y) \in X \times Y \mid y = f(x)\},
\]}%
is subanalytic in $\RR^m \times \RR^n$.

We will also make significant use of Whitney stratifications, which we recall here. Let $\mathcal{I}$ be a partially ordered set with order relation $<$, and $X$ a closed subset of a real smooth manifold $M$. A \emph{Whitney stratification} of $X$ is a locally finite collection of disjoint locally closed smooth submanifolds $S_i \subseteq X$ called \emph{strata} that satisfy:
\begin{enumerate}[\upshape (i)]
\item $X = \bigcup_{i \in \mathcal{I}} S_i$.
\item If $i < j$, then $S_i \cap \overline{S}_j \neq \varnothing$ if and only if $S_i \subseteq \overline{S}_j$, where $\overline{S}_j$ denotes the Euclidean closure of $S_j$.
\item Let $S_{\alpha} \subseteq \overline{S}_{\beta}$. Let $\{x_i\}_{i=1}^{\infty} \subseteq S_{\beta}$ and $\{y_i\}_{i=1}^{\infty} \subseteq S_{\alpha}$ be two sequences converging to the same  $y \in S_{\alpha}$. If the secant lines $\spn\{x_i, y_i\}$ converge to some limiting line $\ell$ and the tangent spaces $\operatorname{T}_{x_i}(S_{\beta})$ converge to some limiting space $\tau$, then $\operatorname{T}_{y}(S_{\alpha}) \subseteq \tau$ and $\ell \subseteq \tau$.
\end{enumerate}
For a proper subanalytic map $f \colon X \to Y$, there are Whitney stratifications of $X$ and $Y$ into analytic manifolds such that $f$ is a \emph{stratified map}, i.e., for each stratum $B \subseteq Y$, $f^{-1}(B)$ is a union of connected components of strata of $X$. We refer the reader to \cite{BierstoneMilman88} for further details, \blue{but the pictures of \emph{Whitney's umbrella} $z^2 - yx^2 = 0$ and its Whitney stratification  in Figure~\ref{fig:umbrella} ought to give readers an intuitive idea.}

\begin{figure}
\includegraphics [viewport = 150 450 450 700, scale = 0.65] {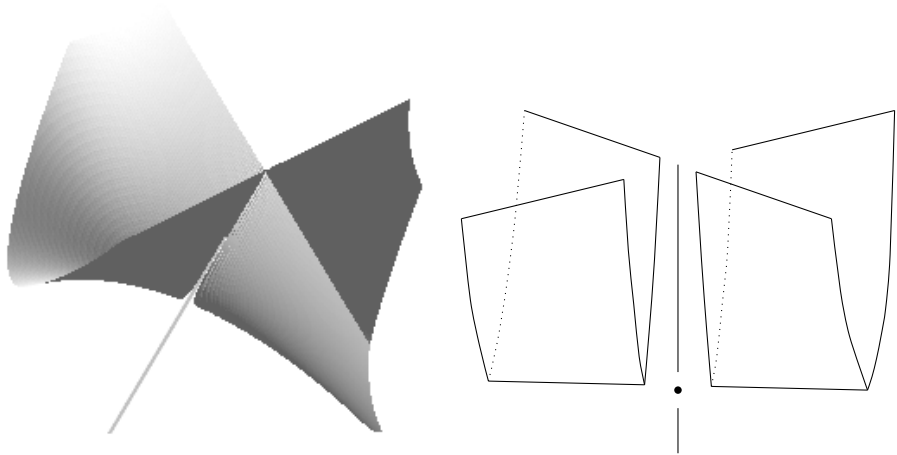}
\hspace*{-6ex}
\begin{tikzpicture}[scale=1.2]
title={Stratification},
\draw [black]  (0.2,-0.08) -- (1.7,-0.6);
\draw [black]  (-1.7,0.6) -- (-0.2,0.08);

\draw [black]  (0.6,0.164) -- (2.4,0.6);
\draw [black]  (-2.4,-0.6) -- (-0.6,-0.164);
\draw [black]  (-2.1, -2.4) -- (-0.37, -2.441);
\draw [black]  (0.37, -2.459) -- (2.1,-2.5);

\draw [black]  plot [smooth] coordinates {(1.7, -0.6) (1.8,-1.4) (1.9,-1.9) (2.1,-2.5)};
\draw [black]  plot [smooth] coordinates {(2.1, -2.5) (2.15,-2.35) (2.2,-2.15) (2.25,-1.9) (2.3,-1.6) (2.35,-1) (2.4,0.6)};

\draw [black]  plot [smooth] coordinates {(-2.4, -0.6) (-2.3,-1.5) (-2.2,-2) (-2.1,-2.4)};
\draw [black, dotted]  plot [smooth] coordinates {(-2.1, -2.4) (-2,-1.96) (-1.9,-1.35) (-1.8,-0.55) (-1.7,0.6)};

\draw [black]  plot [smooth] coordinates {(0.2, -0.08) (0.26,-1.05) (0.32,-1.86) (0.37,-2.459)};
\draw [black, dotted]  plot [smooth] coordinates {(0.37, -2.459) (0.43,-2.13) (0.49,-1.55) (0.55,-0.95) (0.6,0.164)};

\draw [black]  plot [smooth] coordinates {(-0.37, -2.441) (-0.43,-2.13) (-0.49,-1.55) (-0.55,-0.95) (-0.6,-0.164)};
\draw [black]  plot [smooth] coordinates {(-0.2, 0.08) (-0.26,-1.05) (-0.32,-1.86) (-0.37,-2.441)};

\draw [black]  (0, 0) -- (0,-2.3);
\draw [black]  (0, -2.7) -- (0,-3.2);

\node at (0,-2.5) [circle,fill,inner sep=1pt]{};

\end{tikzpicture}
\caption{\emph{Left}: Whitney's umbrella $z^2 - yx^2 = 0$. \emph{Right}: Stratification of Whitney's umbrella.}\label{fig:umbrella}
\end{figure}

The following result is well-known in the semialgebraic setting \cite[Theorem~3.3]{Durfee83} but we need it in a subanalytic setting; fortunately a similar proof yields the required analogous result.
\begin{lemma}\label{lem:subanaldiffae}
Let $U_1$ and $U_2$ be real vector spaces. Let $X \subseteq U_1$ be a $k$-dimensional subanalytic subset and $f \colon X \to U_2$ be a subanalytic map. Then the set of points of $X$ where $f$ is not differentiable is contained in a subanalytic subset  of dimension strictly smaller than $k$.
\end{lemma}
\begin{proof}
Let $\Gamma \subseteq U_1 \times U_2$ be the graph of $f$, and $\pi_1 \colon X \times U_2 \to X$ be the projection. Given a Whitney stratification $\Gamma = \bigcup_{i \in\mathcal{I}} \Gamma_i$ of $\Gamma$, since each $\pi_1(\Gamma_i)$ is subanalytic, there is a Whitney stratification $X = \bigcup_{j \in J} X_j$ such that $\pi_1(\Gamma_i)$ is a union of strata, namely $f\vert_{X_j} \colon X_j \to U_2$ is a subanalytic map whose graph is an analytic submanifold of $U_1 \times U_2$. For simplicity, we denote $f\vert_{X_j}$ by $f_j$ and the graph of $f_j$ by $\Delta_j$. Then the set of points of $X_j$ where $f_j$ is not differentiable is contained in the critical values of $\pi_1 \colon \Delta_j \to X_j$. Hence the set of nondifferentiable points of $f$ in  $X$ is contained in
\[
\{x \in X_j \mid \dim X_j < k\} 
 \cup  \{x \in X_l \mid \dim X_l = k \text{ and } x \text{ is a critical value of } \pi_1 \colon \Delta_l \to X_l\}
\]
which is a subanalytic subset whose dimension is less than $k$. 
\end{proof}

In this article we prove most of our results for \emph{complex analytic varieties}, i.e., defined locally by the common zero loci of finitely many holomorphic functions; although the examples in Section~\ref{sec:app2} are all \emph{complex algebraic varieties}, i.e., those holomorphic functions are polynomials.\footnote{While every complex \emph{projective} analytic variety is algebraic by Chow's theorem,  there exist complex analytic varieties that are not algebraic.} Since a subset of $\mathbb{C}^n$ may be regarded as a subset of $\mathbb{R}^{2n}$, we have the following relations between (semi)algebraic and (semi)analytic sets:
\begin{align*}
\text{complex algebraic variety}\; &\subseteq \; \text{real semialgebraic set}\; \subseteq \; \text{real semianalytic set}, \\
\text{complex algebraic variety} \; &\subseteq \; \text{complex analytic variety} \; \subseteq \; \text{real semianalytic set}.
\end{align*}
So for example any property of real semianalytic sets will automatically be satisfied by a complex analytic variety.  In particular, the following important theorem \cite{GoreskyMacPherson88} also holds true in subanalytic, semialgebraic, or complex algebraic contexts; but for our purpose, it will be stated for complex analytic sets.
\begin{theorem}\label{thm: WhitneyStratification}
Let $M$ be a complex manifold and $Z \subseteq M$ be a closed complex analytic subset. Then there is a Whitney stratification $Z = \bigcup_{i \in \mathcal{I}} S_i$ of $Z$ such that
\begin{enumerate}[\upshape (a)]
\item each $S_i$ is a complex submanifold of $M$;
\item if $S_i \subseteq \overline{S}_j$, there is a vector bundle $E_i \to S_i$, a neighborhood $U_i \subseteq E_i$ of the zero section $S_i$, and a homeomorphsim $h: U_i \to \overline{S}_j$ of $U_i$ to an open subset $h(U_i)$ in $\overline{S}_j$.
\end{enumerate}
\end{theorem}

\subsection*{Notations and terminologies} In the next sections, we will frame our discussion over an abstract real vector space $U$ or an abstract complex vector space $V$. By a semialgebraic or subanalytic subset $X \subseteq U$ we mean that $X$ can be identified with some semialgebraic or subanalytic subset in $\RR^n$ when we identify $U \cong \RR^n$ by fixing a basis of $U$. Likewise, by an algebraic or analytic variety $X \subseteq V$ we mean that $X$ can be identified with some algebraic or analytic variety in $\CC^n$ when we identify $V \cong \CC^n$ by fixing a basis of $V$.

The main reason we prefer such coordinate-free descriptions instead of assuming at the outset that $U = \RR^n$ or $V = \CC^n$ is that we will ultimately be applying the results to cases where $U$ and $V$ have additional structures, e.g., $ \KK^{n_1} \otimes \dots \otimes \KK^{n_d}$, $ \mathsf{S}^d(\KK^n)$, $ \mathsf{\Lambda}^k(\KK^n)$, $ \mathsf{S}^d(\mathsf{\Lambda}^{k}(\KK^n))$, $ \mathsf{S}^{d_1}(\KK^n) \otimes \dots \otimes \mathsf{S}^{d_n}(\KK^n)$, $ \mathsf{\Lambda}^{k_1}(\KK^n) \otimes \dots \otimes \mathsf{\Lambda}^{k_n}(\KK^n)$, etc.

The results in this article will be proved for \emph{general points}. A point in a semianalytic set $X$ is said to be \emph{general} with respect to a property if the subset of points that do not have that property is contained in a semianalytic subset whose dimension is strictly smaller than $\dim X$. In particular, a point  in a real vector space is said to be \emph{general} with respect to a property if the set of points that do not have that property is contained in a real analytic hypersurface. The same  notion applies to a point in a complex vector space, regarded as a real vector space of real dimension twice its complex dimension.

Establishing that a result holds true for all general points is far stronger than the common practice in applied and computational mathematics of establishing it  `with high probability' or `almost surely' or `almost everywhere'. If a result is valid for all general points, then not only do we know that it is valid almost surely/everywhere but also that the invalid points are all limited to a subset of strictly smaller dimension. In fact, when the result is about points in a vector space, which is the case in our article, this subset of invalid points can be described by a single equation that may in principle be determined, so that we know where the invalid points lie.

\section{Best approximation by points in a closed subanalytic set}\label{sec:real}

Let $U$ be an $n$-dimensional real vector space with an inner product $\langle \,\cdot , \cdot \, \rangle$ and corresponding $\ell^2$-norm  $\|\cdot \|$. Let $X \subseteq U$ be a closed subanalytic set. We define the squared distance function $d$ by 
\[
d \colon U \to \RR, \quad
p \mapsto \min_{q \in X} \|p - q\|^2.
\]
For $p \notin X$, a \emph{best $X$-approximation} of $p$  is a minimizer $q \in X$ that attains  $\min_{q \in X}\|p - q\|^2$, and it is customary to write $\argmin_{q \in X}\|p - q\|^2$ for the set of all such minimizers. Let
\begin{equation}\label{eq:nonexist}
\mathcal{C}(X) \coloneqq \{p \in U \setminus X \mid p \text{ does not have a unique best approximation in } X\}.
\end{equation}

We have the following subanalytic analogue of  \cite[Theorem~3.7]{FriedlandStawiska15}.
\blue{\begin{lemma}\label{lem:aeuni}
Let $m, n \in \mathbb{N}$. For any  closed subanalytic set $X \subseteq U$,
the set $\mathcal{C}(X)$ has the following two properties:
\begin{enumerate}[\upshape (i)]
\item $\mathcal{C}(X) \cap \overline{B}(0, m)$ is subanalytic, where $\overline{B}(0, m) \coloneqq \{x \in U \mid \|x\| \le m\}$;
\item $\mathcal{C}(X)$ has dimension strictly less than $n$.
\end{enumerate}
\end{lemma}}%

\begin{proof}
We start by defining the two maps:
\[
\begin{aligned}
f \colon U \times \RR \times X &\to \RR,                               & (p, t, q)                    &\mapsto \|p - q\|^2 - t, \\
g \colon U \times \RR \times X \times (0, \infty) &\to \RR,  & (p, t, q, \varepsilon) &\mapsto t + \varepsilon - \|p - q\|^2,
\end{aligned}
\]
and the two projections:
\begin{align*}
\pi_1 \colon U \times \RR \times X &\to U \times \RR, \\
\pi_2 \colon U \times \RR \times X \times (0, \infty) &\to U \times \RR \times X.
\end{align*}
Let $\Gamma$ be the graph of $d$. Then
\[
\Gamma = \bigl[U \times [0, \infty)\bigr] \cap \bigl[ U \times \RR \setminus \pi_1 \bigl(f^{-1}(- \infty, 0) \bigr)\bigr] \cap \bigl[ \pi_1 \bigl(U \times \RR \times X \setminus \pi_2 (g^{-1}(-\infty, 0] ) \bigr) \bigr]
\]
because 
\begin{align*}
(p, t) &\in  U \times \RR \setminus \pi_1 \bigl(f^{-1}(- \infty, 0)\bigr) \quad \iff \quad \|p - q\|^2 \ge t \text{ for all } q \in X,
\shortintertext{and}
(p, t) &\in  \pi_1\bigl(U \times \RR \times X \setminus \pi_2(g^{-1}(-\infty, 0] )\bigr) \quad \\
&\qquad \iff \quad \text{for any }\varepsilon > 0,\text{ there is a } q \in X\text{ with }\|p - q\|^2 < t + \varepsilon.
\end{align*}
Therefore $d \vert_{\overline{B}(0, m)}$, the squared distance function restricted to $\overline{B}(0, m)$, is a subanalytic function for each $m \in \mathbb{N}$.
By \cite[Theorem~2.1]{FriedlandStawiska15}, the set $\mathcal{C}(X)$ in \eqref{eq:nonexist} comprises precisely the nonsmooth points:
\[
\mathcal{C}(X) = \{p \in U \setminus X \mid d \text{ is not differentiable at } p\}.
\]
Hence, by Lemma~\ref{lem:subanaldiffae}, $\mathcal{C}(X) \cap \overline{B}(0, m)$ is a subanalytic subset with dimension less than $n$. Since
\[
\mathcal{C}(X) = \bigcup_{m=1}^{\infty} \bigl( \mathcal{C}(X) \cap \overline{B}(0, m) \bigr),
\]
$\mathcal{C}(X)$ must also have dimension less than $n$. 
\end{proof}

\section{Best approximation by points in a closed complex variety}\label{sec:cplx}

We will now switch our discussion from $\mathbb{R}$ to $\mathbb{C}$. Let $V$ be an $n$-dimensional complex vector space with a Hermitian inner product $\langle \,\cdot , \cdot\, \rangle$ and corresponding $\ell^2$-norm  $\|\cdot \|$. For any closed irreducible complex analytic variety $X \subseteq V$, again we let $\mathcal{C}(X)$ denote the set of points which do not have a unique best approximation in $X$, i.e., with $V$ in place of $U$ in \eqref{eq:nonexist}. Since $V$ may be regarded as a real vector space of real dimension $2n$, $X$ is naturally a real analytic variety. In particular, best $X$-approximations are  as defined in Section~\ref{sec:real}.

By Lemma~\ref{lem:aeuni}, $\mathcal{C}(X)$ has real dimension strictly less than $2n$ and we easily deduce the following.
\begin{corollary}
Let  $X$ be a closed complex analytic variety in a complex vector space $V$. Then a general $p \in V$ has a unique best $X$-approximation.
\end{corollary}

For any $p \in V \setminus (\mathcal{C}(X) \cup X)$, $p$ has a unique best $X$-approximation, which we will denote by $\pi_{X} (p)$. Thus this gives us a map $\pi_X \colon V \setminus (\mathcal{C}(X) \cup X) \to X$ that sends $p$ to its unique best $X$-approximation $\pi_X (p)$, i.e.,
\[
\| p - \pi_X(p)\|^2 = \min_{q \in X} \| p - q \|^2.
\]
The map $\pi_X$ is also clearly subanalytic when restricted to $\overline{B}(0, m)$. With this observation, we will state  our main result.

\begin{theorem}\label{thm:cmplxapp}
Let $X$ be a closed irreducible complex analytic variety in $V$. Let $Z \subseteq X$ be any complex analytic subvariety with $\dim Z < \dim X$. For a general $p \in V$, its unique best $X$-approximation $\pi_X (p)$ does not lie in $Z$.
\end{theorem}

\blue{Before proving Theorem~\ref{thm:cmplxapp}, it will be instructive to study the simplest case where $\dim X = 1$ and $\dim  Z = 0$, which will provide us with an intuitive idea as to why Theorem~\ref{thm:cmplxapp} holds; more importantly, it will illustrate the difference between working over $\CC$ and working over $\RR$.

\begin{lemma}\label{lem:curve}
Let $X \subseteq \CC^n$ be an irreducible complex analytic curve and $Z \subseteq X$ be a zero-dimensional subvariety. Then for a general $p \in \CC^n$, its unique best $X$-approximation $\pi_X (p)$ does not lie in $Z$.
\end{lemma}

\begin{proof}
We proceed by contradiction. Suppose there is a nonempty open neighborhood $B(p, \varepsilon)$ of $p$ such that $\pi_{X} (B(p, \varepsilon)) \subseteq Z$. Since $Z$ is a collection of points of $X$, we may assume that $0 \in Z$ and $\pi_{X} (B(p, \varepsilon)) = 0$, i.e., 
\[
|x_1|^2 + \dots + |x_n|^2 \le |x_1 - y_1|^2 + \dots + |x_n - y_n|^2
\]
for all $x = (x_1, \dots, x_n) \in B(p, \varepsilon)$ and all $y = (y_1, \dots, y_n) \in X$. By applying elimination theory, in some nonempty open neighborhood $B (0, \eta) \cap X$ of $0$ we may assume $(y_1 (t), \dots, y_n (t))$ is a local parametrization of $X$ such that $y_1 (0) = \dots = y_n (0) = 0$, and each $y_j (t)$ is holomorphic in $t$ around $0$, i.e.,
\[
y_j (t) = \sum_{k=m_j}^{\infty} a_{j, k} t^k,
\]
with $m_j \ge 1$, $a_{j, m_j} \ne 0$, $j =1,\dots,n$. Let
\[
f(x, t) \coloneqq \sum_{j=1}^n \bigl( |x_{j}|^2 - | x_{j} - y_{j} (t)|^2\bigr)
\]
and let $m \coloneqq \min \{m_1, \dots, m_n\}$. Then
\[
f(x, t) = 2 \sum_{j=1}^n \bigl(\Re (a_{j, m_j} t^{m_j} \overline{x}_j) + O(t^{m_j+1}) \bigr) = \Re \bigl(t^m g(x)\bigr) + O(t^{m+1}),
\]
where $g(x) \coloneqq \sum_{j : m_j = m} 2 a_{j, m_j} \overline{x}_j $, i.e., a sum over all $j$ such that $m_j = m$, 
and $O(t^{m+1})$ denotes terms of degrees $\ge m+1$.
Note that  $g(x) \ne 0$ for a general $x$. Let $t^m = \lambda \overline{g(x)}$ for some small $\lambda > 0$. Then $f(x, t) > 0$, which contradicts our assumption that $f(x, t) \le 0$ for all $x = (x_1, \dots, x_n) \in B(p, \varepsilon)$ and all $y = (y_1, \dots, y_n) \in X$. So $\pi_{X} (B(p, \varepsilon))$ cannot be just $0$.
\end{proof}

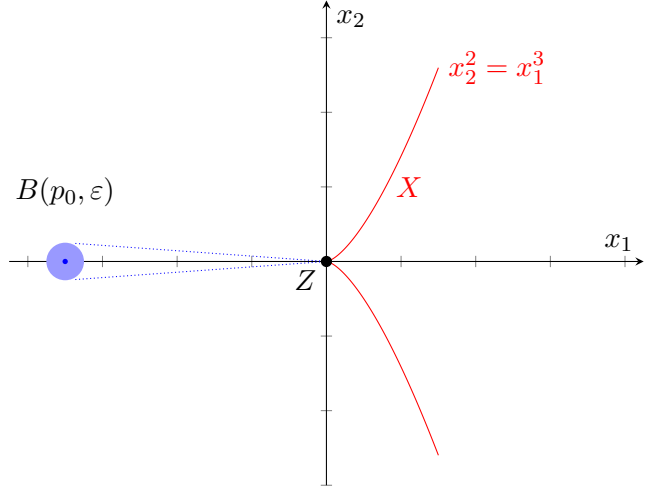
\begin{figure}
\centering
\begin{tikzpicture}[
        point/.style={
            circle,
            fill=blue,
            inner sep=1.5pt,
        },
    ]
        \begin{axis}[
            xmin=-8.5,
            xmax=8.5,
            ymin=-6,
            ymax=7,
            xlabel={$x_1$},
            ylabel={$x_2$},
            scale only axis,
            axis lines=middle,
            domain=-1.912931:3,
            samples=200,
            smooth,
            clip=false,
            axis equal image=true,
            yticklabels={,,},
            xticklabels={,,}
        ]
            \addplot [red] {sqrt(x^3)}
                node[right] {$x_2^2=x_1^3$};
            \addplot [red] {-sqrt(x^3)};

            \fill[blue!40!white] (-7,0) circle (0.25 cm);
            \fill[blue] (-7,0) circle (1 pt);
            
               \node 
                (P) at (-7,0.25 cm)       {};
                \node[above] at (-7, 1.2) {$B(p_0, \varepsilon)$};
               \node 
                (Q) at (-7,-0.25 cm)       {};
               \node [point][black]
                (M) at (0,0)       {};
                \node[left] [black] at (0, -0.5) {$Z$};
                \node[right] [black] at (1.6, 2) {$\color{red}{X}$};
                
            \draw[densely dotted] [blue] (P) -- (M);
            \draw[densely dotted] [blue] (Q) -- (M);
            
        \end{axis}
    \end{tikzpicture}
\caption{Over $\mathbb{R}$, singular point is `exposed.'} \label{fig:real case}
\end{figure}
\begin{figure}
\centering
\includegraphics [scale = 0.6] {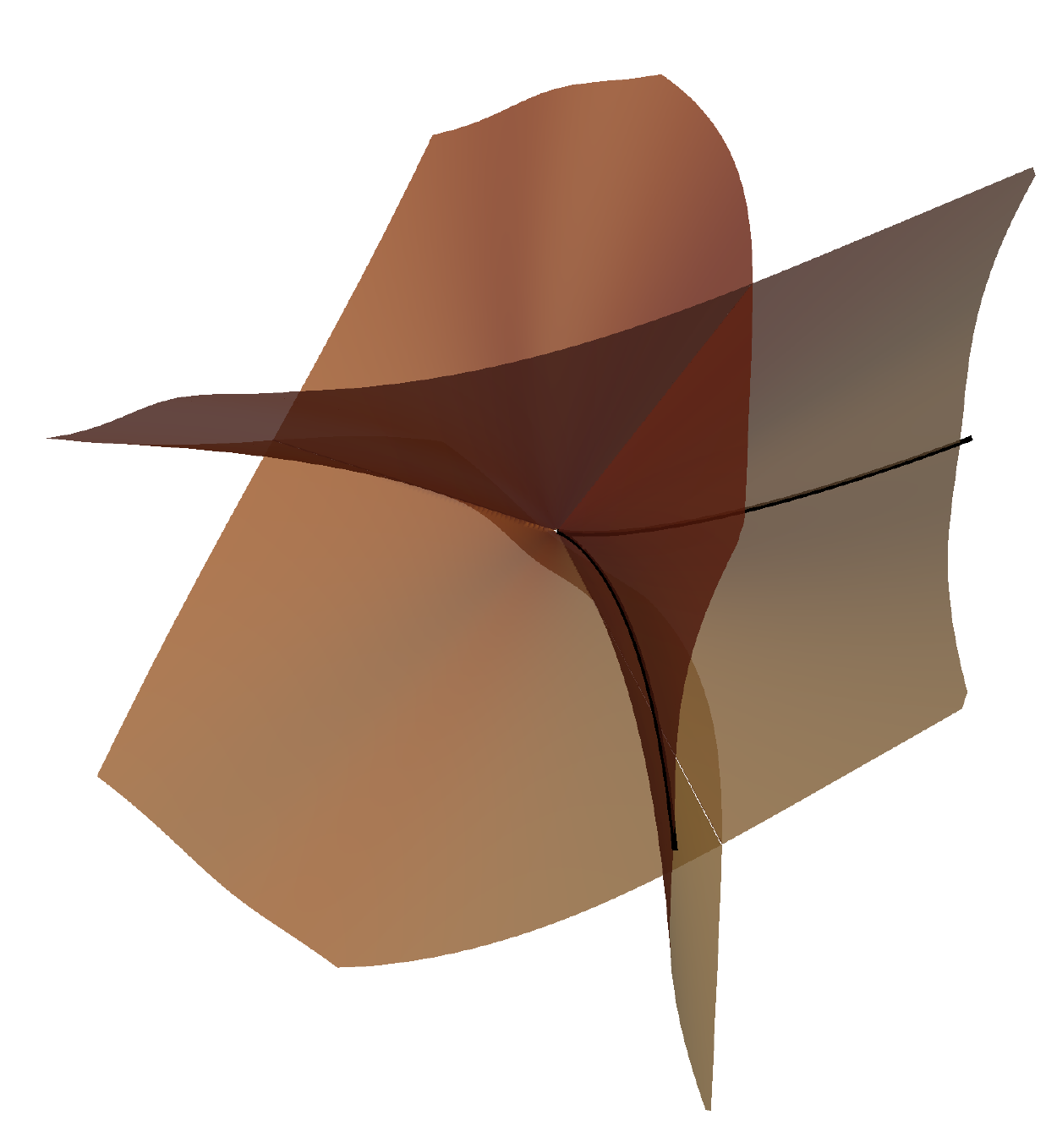}
\caption{Over $\mathbb{C}$, singular point is no longer `exposed.' The figure shows the complex cuspidal curve $x_2^2 =x_1^3$  in $\mathbb{C}^2\simeq\mathbb{R}^4$  projected to $\mathbb{R}^3$.}\label{fig:compcusp}
\end{figure}

An important distinction between $\mathbb{R}$ and $\mathbb{C}$ is that Lemma~\ref{lem:curve} does not hold over $\RR$. An example where
\[
\{p = (x_1, x_2) \in \mathbb{R}^2 \mid q_* \in Z \text{ for all } q_* \in \argmin_{q \in X} \lVert p - q\rVert \}
\]
contains a nonempty open subset is illustrated in Figure~\ref{fig:real case}. Here $X$ is the real cuspidal curve  $\{(x_1, x_2) \in \mathbb{R}^2 \mid x_2^2 = x_1^3\}$, and $Z$ is the cusp $(0, 0)$ in $X$. We can see for any $p_0$ on the negative $x_1$-axis there is a small open neighborhood $B(p_0, \varepsilon)$ such that for any $p \in B(p_0, \varepsilon)$ the best approximation $\pi_X (p)$ of $p$ in $X$ is $Z$. This should be contrasted with the complex cuspidal curve  $\{(x_1, x_2) \in \mathbb{C}^2 \mid x_2^2 = x_1^3\}$, which is a real surface in $\mathbb{C}^2\simeq\mathbb{R}^4$. The projection to $\mathbb{R}^3$ of this curve is the  surface\footnote{The apparent one-dimensional singular locus is an artifact of the projection; the complex curve has only one singular point.} presented in Figure \ref{fig:compcusp}. The real points of the curve form the black cuspidal curve. The crucial observation here is that the complex points of the curve prevent the cusp from being an exposed point, and thus generally preventing it from being a candidate for a best approximant.

Now we are in a position to prove Theorem~\ref{thm:cmplxapp}. A pictorial illustration of the proof is given in Figure~\ref{fig:thm}.}

\begin{figure}
\centering
\begin{tikzpicture}[>=stealth]
%
%
%

\draw [black, thin]  plot [smooth] coordinates {(-2,-1.2) (-1,-0.9) (-0.6,-0.3) (0,0)};
\draw [black, dotted]  plot [smooth] coordinates {(0,0) (1.2,0.85) (2,1.2) (2.5,1.5)};
\draw [red, thick]  plot [smooth] coordinates {(-2,-3.2) (-1,-2.9) (-0.6,-2.3) (0,-2)};
\draw [red, dotted, thick]  plot [smooth] coordinates {(0,-2) (1.2,-1.15) (2,-0.8) (2.5,-0.5)};
\draw [black, thin]  plot [smooth] coordinates {(-2,-1.2) (-1.8,-1.55) (-2,-1.93) (-2.2,-2.5) (-2,-2.95) (-2.1,-3) (-2,-3.2)};
\draw [black, dotted]  plot [smooth] coordinates {(2.5,1.5) (2.7,1.15) (2.5,0.77) (2.3,0.2) (2.5,-0.25) (2.4,-0.3) (2.5,-0.5)};
\draw [black, thin]  plot [smooth] coordinates {(-2,-3.2) (-1,-3.3) (0,-3.55) (1,-3.38) (1.2,-3.57) (2,-3.7)};
\draw [black, thin]  plot [smooth] coordinates {(2,-3.7) (3,-3.2) (3.4,-2.8) (4,-2.5)};
\draw [black, dotted]  plot [smooth] coordinates {(4,-2.5) (5.2,-1.65) (6,-1.3) (6.5,-1)};
\draw [black, dotted]  plot [smooth] coordinates {(2.5,-0.5) (3.5,-0.6) (4.5,-0.85) (5.5,-0.68) (5.7,-0.87) (6.5,-1)};

\draw [blue, thick] plot [smooth] coordinates {(0,-2) (0.8,-2.5) (1,-2.3) (2,-2.6) (3,-2.35) (4,-2.5)};
\draw [blue, thick] plot [smooth] coordinates {(0,-2) (0.2,-1.5) (-0.3,-1) (0,-0.5) (0.1,-0.2) (0,0)};

\draw [black, thin] (7,1.15) -- (7,-4.3);
\draw [black, thin] (7,1.15) -- (-1.2,2);


%
\node [red] at (0,-2) [circle,fill,inner sep=1.2pt]{};

\fill[blue!24!white] (4.2,0.3) circle (0.6cm);
\node [black] at (4.2,0.3) [circle,fill,inner sep=0.9pt] {};

\node [black] at (4.4, -1.7) {$X$};
\node [black] at (6.5, 0.8) {$W$};
\draw [<-] (4.4,0.1) -- (5,-0.3);
\node [black, right] at (5, -0.3) {$B(p, \epsilon) \cap W$};
\draw [<-] (3,-2.5) -- (5,-3);
\node [black, right] at (5, -3) {$X \cap W$};
\draw [<-] (-1,-2.7) -- (-3,-2.3);
\node [black, left] at (-3, -2.3) {$Z$};
\draw [<-] (-0.1,-1.9) -- (-1.6,0.6);
\node [black, left] at (-1.6,0.6) {$Z \cap W$};

%

                
\end{tikzpicture}
\caption{Pictorial illustration of the proof of Theorem~\ref{thm:cmplxapp}.} \label{fig:thm}
\end{figure}
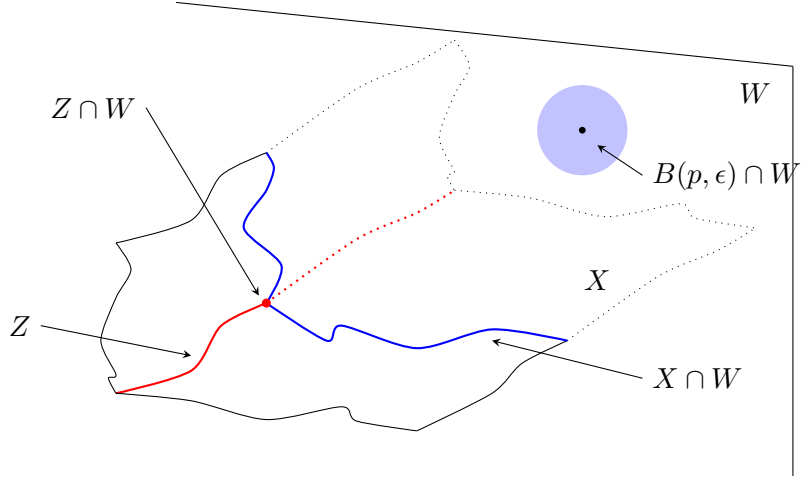

\begin{proofof}{Theorem~\ref{thm:cmplxapp}}
We proceed by contradiction. Suppose there is a nonempty open subset $\mathcal{O} \subseteq V \setminus (\mathcal{C}(X) \cup X)$ such that for any $p \in \mathcal{O}$, the best $X$-approximation $\pi_X (p)$ lies in $Z$. Without loss of generality, we may assume that $Z$ is a minimum subvariety of $X$ containing $\pi_X (\mathcal{O})$ in the sense that any subvariety containing $\pi_X (\mathcal{O})$ has dimension at least $\dim_{\CC} Z$. For simplicity, we choose a basis of $V$ so that we may write $u = (u_1,\dots, u_n)$, $v = (v_1, \dots, v_n) \in V$, and the inner product is given by $\langle u, v \rangle = u_1\overline{v}_1 + \dots + u_n \overline{v}_n$ under this basis. With this assumption, the norm is $\|v\| =( |v_1|^2 + \dots + |v_n|^2)^{1/2}$. 

First we assume $\dim Z > 0$. By taking a Whitney stratification, there is a nonsingular point $z \in Z$ and an open ball $B(z, \rho) \subseteq V$ around $z$, such that $B(z, \rho) \cap Z$ is a connected complex submanifold in $V$, and $\pi_X^{-1} (B(z, \rho) \cap Z)$ contains a nonempty open subset $\mathcal{O}' \subseteq \mathcal{O}$. For notational convenience, we  will still denote $B(z, \rho) \cap Z$ by $Z$ and $\mathcal{O}'$ by $\mathcal{O}$.

Given $p \in \mathcal{O}$, let $B(p, \varepsilon) \subseteq \mathcal{O}$ be a small open ball of $p$, and let $q = \pi_X (p)$. Then
\[
\frac{d}{dt} \langle p - q(t), p - q(t) \rangle\Big\vert_{t = 0} = 0
\]
for any real analytic curve $q(t) \subseteq Z$ with $q(0) = q$, which implies that
\[
\Re \langle p - q, v \rangle = 0 \quad \text{for any} \; v \in \operatorname{T}_q(Z),
\]
where $\operatorname{T}_q(Z)$ is the tangent space of $Z$ at $q$. Without loss of generality, let $q = 0 \in V$, and let
\[
W \coloneqq \{ x \in V \mid \Re \langle x - q, v \rangle = 0 \text{ for any }v \in \operatorname{T}_q(Z)\}.
\]
Since $\operatorname{T}_q(Z) = \operatorname{T}_0(Z)$ is a complex vector space, we in fact have
\[
W = \{ x \in V \mid \langle x, v \rangle = 0 \text{ for any } v \in \operatorname{T}_0(Z)\},
\]
i.e., $W$ is a complex vector subspace of $V$ whose complex codimension is $\dim_{\CC} Z$. Thus
\[
W \cap Z \cap B(0, \delta) =\{ 0\}
\]
for some small open ball $B(0, \delta) \subseteq V$. Since $X$ is irreducible and the set of nonsingular points of $X$ is dense in $X$, $(X \setminus Z) \cap B(0, \delta) \neq \varnothing$. By Theorem~\ref{thm: WhitneyStratification}, $Z$ has a small tubular neighborhood in $X$. Thus $W \cap X \cap B(0, \delta)$ is of positive dimension. Hence it suffices to show that $\pi_{W \cap X \cap B(0, \delta)} (W \cap B(p, \varepsilon))$ cannot be just $\{0\}$.

We have reduced the problem to showing that in the complex vector space $W$, given a positive-dimensional complex analytic variety $W \cap X \cap B(0, \delta)$ and a point $0 \in W \cap X \cap B(0, \delta)$,
\[
\pi_{W \cap X \cap B(0, \delta)} (W \cap B(p, \varepsilon)) \ne \{ 0\}
\]
for a nonempty open ball $W \cap B(p, \varepsilon)$ in $W$. Again for notational simplicity, we will still use $X$ to denote $W \cap X \cap B(0, \delta)$, and $B(p, \varepsilon)$ to denote $W \cap B(p, \varepsilon)$. One should note that  this is in fact the case $\dim Z = 0$.

Pick an irreducible complex analytic curve $C \subseteq X$ passing through $0$. \blue{By Lemma~\ref{lem:curve}, we have that $\pi_{C} (B(p, \varepsilon)) \ne \{0\}$, and thus $\pi_X (B(p, \varepsilon)) \ne \{0\}$.}
\end{proofof}

Since the singular locus of a complex analytic variety is a complex analytic subvariety of smaller dimension \cite[Chapter~0]{GriffithsHarris94}, we obtain the following corollary of Theorem~\ref{thm:cmplxapp}. Given its important implications in this article, we label it a theorem.
\begin{theorem}\label{thm:nonsing}
Let $X$ be a closed irreducible complex analytic variety in $V$. For a general $p \in V$, its best $X$-approximation $\pi_X (p)$ is a nonsingular point of $X$.
\end{theorem}

\blue{In fact, the requirement in Theorem~\ref{thm:nonsing} that $X$ is irreducible can be relaxed, allowing for reducible varieties, and even constructible sets.
\begin{corollary}\label{cor:cmplxappred}
Let $X \subseteq V$ be a closed complex analytic variety, and $X = X_1 \cup \cdots \cup X_k$ be the decomposition of $X$ into irreducible components. Let $Z \subseteq X$ be any complex analytic subvariety such that $\dim Z < \dim X$ and $Z \cap X_j$ is not Zariski dense in $X_j$ for all $j=1,\dots, k$. Then for a general $p \in V$, its unique best $X$-approximation $\pi_X (p)$ does not lie in $Z$.
\end{corollary}

\begin{proof}
Apply Theorem~\ref{thm:cmplxapp} to each $X_j$ and $Z \cap X_j$.
\end{proof}

\begin{corollary}\label{cor:imofpoly}
Let $f:X\rightarrow V$ be a morphism from a complex affine, possibly reducible, algebraic variety $X$ to a complex vector space $V$. Then for a general $p\in V$, its unique best $\overline{f(X)}$-approximation $\pi_{\overline{f(X)}}(p)$ lies in $f(X)$.   
\end{corollary}
\begin{proof}
By Chevalley's theorem, $f(X)$ is a constructible set \cite[Theorem~3.16]{Harris95:gtm133}. We set $Z:=\overline{\overline{f(X)}\setminus f(X)}$ and apply Corollary~\ref{cor:cmplxappred}.
\end{proof}
}

\section{\blue{Applications to nonlinear approximations}}\label{sec:app1}

We will now apply the results in Section~\ref{sec:cplx} back to the functional approximation problems described in Section~\ref{sec:intro}. Since our results require that we work over $\CC$ (in fact, they are false over $\RR$), we will write $L^2(\Omega) = L^2_\CC(\Omega)$ in the following. 

We caution the reader that when applied to the best $r$-term approximation problem in \eqref{eq:main}, the set $X$ in Theorem~\ref{thm:nonsing} is \emph{not} the set of $r$-term approximant
\[
\Sigma^{\circ}_r(Y)  = \{f \in L^2(\Omega)  \mid f = f_1 + \dots + f_r \text{ for some } f_1, \dots, f_r \in Y\},
\]
but  the closure of this set, which we denoted by $\Sigma_r(Y)$. Since it is a closure, $X =\Sigma_r(Y)$ is automatically closed; and requiring it be semianalytic is a very mild condition. In fact, we are unaware of any nonpathological example where this does not hold. 
\begin{corollary}\label{cor:rterm}
Let $f \in L^2(\Omega)$ and $Y \subseteq L^2(\Omega) $.
If $\Sigma_r(Y)$ is semianalytic, then the best $r$-term approximation problem
\begin{equation}\label{eq:maininf}
\inf_{f_1,\dots,f_r \in Y} \lVert f - (f_1 + \dots + f_r) \rVert = \inf_{\varphi \in \Sigma^{\circ}_r(Y)} \lVert f - \varphi \rVert 
\end{equation}
has a solution with probability one, i.e., the infimum is attained by a point in $\Sigma^{\circ}_r(Y)$ with probability one. 
\end{corollary}
\begin{proof}
The infimum in \eqref{eq:maininf}  may not be attained over  $\Sigma^{\circ}_r(Y)$ because this is in general not a closed set. However, the infimum must always be attained in its closure $\Sigma_r(Y)$ even when this set is unbounded --- the reason being that for a fixed $f$, the map $\Sigma_r(Y) \to \mathbb{R}$, $ \varphi \to \lVert f - \varphi \rVert$ is a \emph{coercive} map \cite{LC} and has bounded sublevel sets. Now by Theorem~\ref{thm:nonsing}, the minimizer $\varphi_* \in \Sigma_r(Y)$ is a smooth point with probability one and therefore we  have $\varphi_* \in \Sigma^{\circ}_r(Y)$ with probability one.
\end{proof}

We will next apply our result to the join approximation problem
\begin{equation}\label{eq:genjoin}
\min_{f_{i1},\dots,f_{ir_i} \in Y_i} \lVert f - (f_{11} + \dots + f_{1r_1}) -  (f_{21} + \dots + f_{2r_2}) - \dots- (f_{k1} + \dots + f_{kr_k}) \rVert.
\end{equation}
Note that this is a generalization of \eqref{eq:join} to $k \ge 2$.

Let $X_1,\dots,X_k \subseteq L^2(\Omega)$ and define their \emph{join} as
\begin{equation}\label{eq:joinvar}
J (X_1, \dots, X_k) \coloneqq \{ \varphi_1 + \dots + \varphi_k \in L^2(\Omega) \mid \varphi_i \in X_i, \; i=1,\dots, k\}.
\end{equation}
We observe that \eqref{eq:genjoin} may be reexpressed in the form \eqref{eq:joininf} below with $X_i = \Sigma^\circ_{r_i}(Y_i)$, $i =1,\dots,k$.
\begin{corollary}\label{cor:joinapprox}
Let $f \in L^2(\Omega)$ and $Y_1,\dots, Y_k \subseteq L^2(\Omega) $.
If the join
\[
J\bigl(\Sigma_{r_1}(Y_1),\dots, \Sigma_{r_k}(Y_k)\bigr)
\]
is semianalytic, then the join approximation problem
\begin{equation}\label{eq:joininf}
\inf_{\varphi_1 \in \Sigma_{r_1}^\circ(Y_1),\dots,\varphi_k \in \Sigma_{r_k}^\circ(Y_k)} \lVert f - (\varphi_1 + \dots + \varphi_r) \rVert 
\end{equation}
has a solution with probability one, i.e., the infimum is attained by a point in $J\bigl(\Sigma^{\circ}_{r_1}(Y_1),\dots,\Sigma^{\circ}_{r_k}(Y_k)\bigr)$ with probability one. 
\end{corollary}
\begin{proof}
Observe that $J (X_1, \dots, X_k) $ is an image of the map
\[
J : L^2(\Omega) \times \dots \times L^2(\Omega) \to L^2(\Omega) , \quad (x_1,\dots,x_k) \mapsto x_1 + \dots + x_k,
\]
as $J (X_1, \dots, X_k) = J(X_1 \times \dots \times X_k) $. Since $J$ is a closed map, $J (X_1, \dots, X_k) $   is a closed set if $X_1, \dots, X_k$ are closed sets. Also,
\[
\overline{J (X_1, \dots, X_k) } = J (\overline{X}_1, \dots, \overline{X}_k).
\]
and so
\[
\overline{ J\bigl(\Sigma_{r_1}^\circ(Y_1),\dots, \Sigma_{r_k}^\circ(Y_k)\bigr)} =  J\bigl(\Sigma_{r_1}(Y_1),\dots, \Sigma_{r_k}(Y_k)\bigr).
\]
Applying Theorem~\ref{thm:nonsing} to $X = J\bigl(\Sigma_{r_1}(Y_1),\dots, \Sigma_{r_k}(Y_k)\bigr)$ then yields the required result. 
\end{proof}

Again we stress that it is crucial in Corollaries~\ref{cor:rterm} and \ref{cor:joinapprox} that the functions be complex-valued, the results fail when we replace $\mathbb{C}$ by $\mathbb{R}$.
We will discuss specific cases of Corollaries~\ref{cor:rterm} and \ref{cor:joinapprox} in the next \blue{two sections}.

\section{Applications to separable approximations}\label{sec:app2}

We will now apply the results in Section~\ref{sec:cplx} to more specific classes of functional approximation problems. As in Section~\ref{sec:app1}, given that our results only hold over $\CC$, we write $L^2(\Omega) = L^2_\CC(\Omega)$ throughout this section. A point to recall is that
\[
L^2(\Omega_1 \times \dots \times \Omega_d) = L^2(\Omega_1) \otimes \dots \otimes L^2(\Omega_d)
\]
and so multivariate functions are in fact tensors. In many, if not most applications, the multivariate target function (i.e., the function to be approximated) possesses various forms of symmetries, the most common being \emph{full} invariance or skew-invariance under all permutations of arguments:
\[
f(x_{\tau(1)},\dots, x_{\tau(d)}) = f(x_1,\dots, x_d) \quad \text{or} \quad g(x_{\tau(1)},\dots, x_{\tau(d)}) = \operatorname{sgn}(\tau)  g(x_1,\dots, x_d),
\]
for all $\tau \in \mathfrak{S}_d$. In other words, $f$ is a symmetric tensor and $g$ an alternating tensor:
\[
f \in \mathsf{S}^d( L^2(\Omega) )\qquad \text{or} \qquad g \in \mathsf{\Lambda}^d(L^2(\Omega)).
\]
Note that in these cases, we need $\Omega_1 = \dots = \Omega_d = \Omega$ in order to permute arguments. There are many types of \emph{partial} symmetries as well. For example, we may have functions that satisfy
\begin{align*}
f(x_{\rho(1)},\dots, x_{\rho(d)}, y_{\tau(1)},\dots, y_{\tau(k)}) &= f(x_1,\dots, x_d, y_1,\dots,y_k), \\ g(x_{\rho(1)},\dots, x_{\rho(d)},y_{\tau(1)},\dots, y_{\tau(k)}) &= \operatorname{sgn}(\rho) \operatorname{sgn}(\tau)  g(x_1,\dots, x_d, y_1,\dots,y_k)
\end{align*}
for all $\rho \in \mathfrak{S}_d$ and $\tau \in \mathfrak{S}_k$. These correspond respectively to functions
\[
f \in \mathsf{S}^d( L^2(\Omega_1) ) \otimes \mathsf{S}^k( L^2(\Omega_2) ) \qquad \text{and} \qquad g \in \mathsf{\Lambda}^d(L^2(\Omega_1)) \otimes \mathsf{\Lambda}^k(L^2(\Omega_2)).
\]
There are also more intricate symmetries. For example, a complex-valued function $f(X)$ of a matrix variable $X = (x_{ij})_{i,j=1}^{k,d}$ that satisfies
\begin{multline*}
f(x_{\rho_1(1), \tau(1)},\dots, x_{\rho_1(k), \tau(1)}, \dots, x_{\rho_d(1), \tau(d)},\dots, x_{\rho_d(k), \tau(d)}) \\
= \operatorname{sgn}(\rho_1) \dots \operatorname{sgn}(\rho_d)  f(x_{1, 1},\dots, x_{k, 1}, \dots, x_{1, d},\dots, x_{k, d})
\end{multline*}
for all $\rho_1, \dots, \rho_d \in \mathfrak{S}_k$ and $\tau \in \mathfrak{S}_d$. This corresponds to $f \in  \mathsf{S}^d(\mathsf{\Lambda}^k( L^2(\Omega) ))$. We will see more examples later.

Each of these classes of functions comes with its own nonlinear $r$-term approximation problem that preserves the respective symmetries or skew-symmetries, full or partial. For  example, for $f \in L^2(\Omega \times \dots \times \Omega) $, we may  just want
\begin{equation}\label{eq:seg}
\inf \biggl\lVert f - \sum_{i=1}^r \varphi_{1,i} \otimes \dots \otimes \varphi_{d,i} \biggr\rVert,
\end{equation}
as discussed in Section~\ref{sec:intro}, but for a symmetric $f \in \mathsf{S}^d( L^2(\Omega) )$ or an alternating $g \in \mathsf{\Lambda}^d(L^2(\Omega))$, the more natural approximation problems are respectively
\begin{equation}\label{eq:chow-grass}
\inf \biggl\lVert f - \sum_{i=1}^r \varphi_{1,i} \circ \dots \circ \varphi_{d,i} \biggr\rVert \qquad \text{or} \qquad
\inf \biggl\lVert g - \sum_{i=1}^r \varphi_{1,i} \wedge \dots \wedge \varphi_{d,i} \biggr\rVert,
\end{equation}
with $\varphi_{j,i}  \in L^2(\Omega)$, $i=1,\dots,r$, $j=1,\dots,d$.  In fact, there are other natural options for a symmetric $f$, where we may instead seek an approximation of the form
\begin{equation}\label{eq:ver}
\inf \biggl\lVert f - \sum_{i=1}^r \varphi_{i}^{\otimes d} \biggr\rVert,
\end{equation}
with $\varphi_i  \in L^2(\Omega)$, $i=1,\dots,r$. The symbols $\circ$ and $\wedge$ denote  symmetric product and exterior product respectively; for those who have not encountered them, these and other notations will be defined in Section~\ref{sec:tensor}. Then in Section~\ref{sec:approx}, we will describe other classes of  $r$-term  approximation problems for functions possessing more complicated symmetries than the ones in \eqref{eq:seg}, \eqref{eq:chow-grass}, \eqref{eq:ver}. All these approximation problems have two features in common:
\begin{enumerate}[\upshape (i)]
\item\label{item:variety} Each of them is associated with a complex algebraic variety; e.g., \eqref{eq:seg} with the Segre variety, \eqref{eq:chow-grass} with the Chow variety and Grassmann variety respectively, \eqref{eq:ver} with the Veronese variety.

\item\label{item:no} All of them may fail to have a solution; e.g., there exist target functions where the infima in \eqref{eq:seg}, \eqref{eq:chow-grass}, \eqref{eq:ver} cannot be attained.
\end{enumerate}
We will later see how our results in Section~\ref{sec:cplx} combined with \eqref{item:variety} allow us to rectify \eqref{item:no} \blue{to the extent guaranteed by Theorem~\ref{thm:cmplxapp}, i.e., for any general $f$.}

We would like to mention a slightly different, perhaps more common, alternative for framing the function approximation problems above. As we described in Section~\ref{sec:intro}, for computational purposes, $L^2(\Omega_i)$'s are all finite-dimensional (e.g., when $\Omega_i$'s are all finite sets) and in this case, $L^2(\Omega_i) \cong \mathbb{C}^{n_i}$ with a choice of basis, and thus
\begin{align*}
L^2(\Omega_1 \times \Omega_2) &=   L^2(\Omega_1) \otimes L^2(\Omega_2) \cong \mathbb{C}^{n_1} \otimes \mathbb{C}^{n_2} \cong \mathbb{C}^{n_1 \times n_2,}\\
L^2(\Omega_1 \times \Omega_2\times \Omega_3) &=   L^2(\Omega_1) \otimes L^2(\Omega_2)\otimes L^2(\Omega_3) \cong \mathbb{C}^{n_1} \otimes \mathbb{C}^{n_2} \otimes \mathbb{C}^{n_3} \cong \mathbb{C}^{n_1 \times n_2 \times n_3},\\
\shortvdotswithin{=}
L^2(\Omega_1 \times \dots \times \Omega_d) &=  L^2(\Omega_1) \otimes\dots \otimes L^2(\Omega_d) \cong \mathbb{C}^{n_1} \otimes \dots \otimes \mathbb{C}^{n_d}  \cong \mathbb{C}^{n_1 \times \dots \times n_d}.
\end{align*}
The set $\mathbb{C}^{n_1 \times \dots \times n_d}$ denotes the vector space of $d$-dimensional hypermatrices 
\[
(a_{i_1,\dots, i_d})_{i_1,\dots, i_d= 1}^{n_1,\dots, n_d}
\]
(when $d =2$, this reduces to a usual matrix in linear algebra), which of course is nothing more than a convenient way to represent a function $f : \Omega_1\times \dots \times \Omega_d \to \mathbb{C}$ by storing its value
\[
f(i_1,\dots, i_d) = a_{i_1,\dots, i_d}
\]
at $(i_1,\dots, i_d) \in \Omega_1\times \dots \times \Omega_d$, assuming that $\Omega_i$'s are all finite sets.  In fact, practitioners in computational  mathematics overwhelmingly regard a tensor as such a hypermatrix. The symmetries described earlier carry verbatim to hypermatrices with the permutations acting on the indices. 

We may either frame our results in this section in the form of function approximations or in the form of  tensor approximations (or more accurately, matrix/hypermatrix approximations), but instead of favoring one over the other, we would simply resort to stating them for abstract vector spaces. So implicitly, $V = L^2(\Omega_1 \times \dots \times \Omega_d)$ for function approximations and $V = \mathbb{C}^{n_1 \times \dots \times n_d}$ for tensor approximations.

\subsection{Segre/Veronese/Grassmann varieties and friends}\label{sec:tensor}

We begin by reviewing some basic tensor constructions. The varieties that we define will be subsets of tensor spaces constructed via one or more of the following ways. We write $V_1 \otimes \dots \otimes V_d$  for the tensor product of  vector spaces $V_1, \dots, V_d$. A rank-one tensor is a nonzero decomposable tensor $v_1 \otimes \dots \otimes v_d$, where $v_i \in V_i$.  When $V_1 = \dots = V_d = V$, we use the abbreviations $V^{\otimes d} = V\otimes \dots \otimes V$ and $v^{\otimes d} = v \otimes \dots \otimes v$. In this case, the symmetric group $\mathfrak{S}_d$  acts on rank-one tensors by
\[
\tau(v_1 \otimes \dots \otimes v_d) = v_{\tau(1)} \otimes \dots \otimes v_{\tau(d)},
\]
and this action extends linearly to an action on $V^{\otimes d}$. The subspaces
\begin{align*}
\mathsf{S}^d(V) &= \{T \in V^{\otimes d} \mid \tau(T) = T \text{ for all } \tau \in \mathfrak{S}_d\}, \\
\mathsf{\Lambda}^d(V) &= \{T \in V^{\otimes d} \mid \tau(T) = \operatorname{sgn}(\tau) T \text{ for all } \tau \in \mathfrak{S}_d\},
\end{align*}
are called the spaces of \emph{symmetric $d$-tensors} and \emph{alternating $d$-tensors} respectively. We may construct tensor spaces with more intricate symmetries and skew-symmetries by combining these, e.g., $\mathsf{S}^{d}(V) \otimes \mathsf{S}^{k}(V)$,  $\mathsf{\Lambda}^{d}(V)  \otimes \mathsf{\Lambda}^{k}(V)$, $\mathsf{S}^d(\mathsf{\Lambda}^k(V))$, $\mathsf{S}^d(\mathsf{S}^k(V))$, etc.

The \emph{symmetric  product} and \emph{alternating product} of $v_1,\dots,v_d \in V$ are respectively defined by
\begin{align*}
v_1 \circ \dots \circ v_d &\coloneqq \frac{1}{d!} \sum_{\tau \in \mathfrak{S}_d} v_{\tau(1)} \otimes \dots \otimes v_{\tau(d)} \in \mathsf{S}^d(V),\\
v_1 \wedge \dots \wedge v_d &\coloneqq \frac{1}{d!} \sum_{\tau \in \mathfrak{S}_d}  \operatorname{sgn}(\tau)  v_{\tau(1)} \otimes \dots \otimes v_{\tau(d)} \in \mathsf{\Lambda}^d(V).
\end{align*}

In finite dimensions, each of the approximation problems we have encountered  earlier and yet others that we will see in Section~\ref{sec:approx} comes with an associated complex algebraic variety. These varieties are usually defined as complex projective varieties, i.e., subsets of projective spaces defined by the common zero loci of a finite collection of homogeneous polynomials, and we will not deviate from this standard practice in our definitions below:

\paragraph{Segre variety} This is the image $\sigma(\PP V_1 \times \dots \times \PP V_d)$ of the Segre embedding:
\[
\sigma \colon \PP V_1 \times \dots \times \PP V_d \to \PP (V_1 \otimes \dots \otimes V_d), \quad
([v_1], \dots, [v_d]) \mapsto [v_1 \otimes \dots \otimes v_d].
\]

\paragraph{Veronese variety}  This is the image $\nu (\PP V)$ of the Veronese embedding:
\[
\nu \colon \PP V \to \PP \mathsf{S}^d(V), \quad
[v] \mapsto [v^{\otimes d}].
\]

\paragraph{Chow variety}  This is the image\footnote{This is usually called the \emph{Chow variety of zero cycles} in $\PP V^*$. It is a subvariety of $\PP \mathsf{S}^d(V)$ whose affine cone is the set of degree-$d$ homogeneous polynomials on the dual space $V^*$ that can be decomposed into a product of linear forms \cite[Chapter~4, Proposition~2.1]{GKZ94}.} $\Ch_d(V) \coloneqq \kappa \bigl((\PP V)^d\bigr)$ of the Chow map:
\[
\kappa \colon \PP V \times \dots \times \PP V  \to \PP \mathsf{S}^d(V), \quad
([v_1], \dots,[v_d]) \mapsto [v_1 \circ \dots \circ v_d] .
\]

\paragraph{Grassmann variety} This is the image  $\psi\bigl(\Gr_k(V)\bigr)$ of the \emph{Grassmannian} $\Gr_k(V)$, the set of $k$-dimensional linear subspaces of $V$, under the \emph{Pl\"ucker embedding}:
\[
\psi \colon \Gr_k(V) \to \PP \mathsf{\Lambda}^k(V), \quad \spn\{ v_1, \dots, v_k \}  \mapsto [v_1 \wedge \dots \wedge v_k].
\]

\paragraph{Segre--Veronese variety} This is the image  $\sigma\nu (\PP V_1 \times \dots \times \PP V_m)$ of the Segre--Veronese embedding:
\begin{align*}
\sigma\nu \colon \PP V_1 \times \dots \times \PP V_m &\to \PP \bigl(\mathsf{S}^{d_1}(V) \otimes \dots \otimes \mathsf{S}^{d_m}(V)\bigr), \\
([v_1], \dots, [v_m]) &\mapsto [v_1^{\otimes d_1} \otimes \dots \otimes v_m^{\otimes d_m}].
\end{align*}

\paragraph{Segre--Chow variety} This is the image $\sigma\kappa \bigl(\Ch_{d_1}(V) \times \dots \times \Ch_{d_m}(V)\bigr)$  of the Segre--Chow map:
\begin{align*}
\sigma\kappa \colon \Ch_{d_1}(V) \times \dots \times \Ch_{d_m}(V) &\to  \PP \bigl(\mathsf{S}^{d_1}(V) \otimes \dots \otimes \mathsf{S}^{d_m}(V)\bigr), \\
([v_{1,1} \circ \dots \circ v_{d_1,1}], \dots, [v_{1,m} \circ \dots \circ v_{d_m,m}]) &\mapsto\\
[(v_{1,1} \circ \dots \circ v_{d_1,1}) &\otimes \dots \otimes (v_{1,m} \circ \dots \circ v_{d_m,m})].
\end{align*}

\paragraph{Segre--Grassmann variety} This is the image  $\sigma\psi \bigl(\Gr_{k_1}(V) \times \dots \times \Gr_{k_m}(V)\bigr)$  of the Segre--Grassmann map:
\begin{align*}
\sigma\psi \colon \Gr_{k_1}(V) \times \dots \times \Gr_{k_m}(V) &\to \PP \bigl(\mathsf{\Lambda}^{k_1}(V) \otimes \dots \otimes \mathsf{\Lambda}^{k_m}(V)\bigr), \\
([v_{1,1} \wedge \dots \wedge v_{k_1,1}], \dots, [v_{1,m} \wedge \dots \wedge v_{k_m,m}]) &\mapsto \\
[(v_{1,1} \wedge \dots \wedge v_{k_1,1}) &\otimes \dots \otimes (v_{1,m} \wedge \dots \wedge v_{k_m,m})].
\end{align*}

\paragraph{Veronese--Chow variety} This is the image $\nu\kappa \bigl(\Ch_d(V)\bigr)$  of the Veronese--Chow map:
\begin{align*}
\nu\kappa \colon \Ch_d(V) \to \PP \bigl(\mathsf{S}^k(\mathsf{S}^{d}(V))\bigr), \quad
[v_{1} \circ \dots \circ v_{d}] \mapsto [(v_{1} \circ \dots \circ v_{d})^{\otimes k}].
\end{align*}

\paragraph{Veronese--Grassmann variety} This is the image $\nu\psi\bigl(\Gr_k(V)\bigr)$ of the Veronese--Grassmann map:
\begin{align*}
\nu\psi \colon \Gr_k(V) \to \PP \bigl(\mathsf{S}^d(\mathsf{\Lambda}^{k}(V))\bigr), \quad
[v_{1} \wedge \dots \wedge v_{k}] \mapsto [(v_{1} \wedge \dots \wedge v_{k})^{\otimes d}].
\end{align*}

\paragraph{Segre--Veronese--Chow variety} This is the image $\sigma\nu\kappa\bigl(\Ch_{d_1}(V) \times \dots \times \Ch_{d_m}(V)\bigr)$ of the Segre--Veronese--Chow map:
\begin{align*}
\sigma\nu\kappa \colon \Ch_{d_1}(V) \times \dots \times \Ch_{d_m}(V) \to \PP \bigl(\mathsf{S}^{k_1} &(\mathsf{S}^{d_1}(V)) \otimes \dots \otimes \mathsf{S}^{k_m}(\mathsf{S}^{d_m}(V))\bigr), \\
([v_{1,1} \circ \dots \circ v_{d_1,1}], \dots, [v_{1,m} \circ \dots \circ v_{d_m,m}]) &\mapsto  \\
[(v_{1,1} \circ \dots \circ v_{d_1,1})^{\otimes k_1} &\otimes \dots \otimes (v_{1,m} \circ \dots \circ v_{d_m,m})^{\otimes k_m}].
\end{align*}

\paragraph{Segre--Veronese--Grassmann variety} This is $\sigma\nu\psi\bigl(\Gr_{k_1}(V) \times \dots \times \Gr_{k_m}(V)\bigr)$,  image of the Segre--Veronese--Grassmann map:
\begin{align*}
\sigma\nu\psi \colon \Gr_{k_1}(V) \times \dots \times \Gr_{k_m}(V) \to \PP \bigl(\mathsf{S}^{d_1}&(\mathsf{\Lambda}^{k_1}(V)) \otimes \dots \otimes \mathsf{S}^{d_m}(\mathsf{\Lambda}^{k_m}(V))\bigr), \\
([v_{1,1} \wedge \dots \wedge v_{k_1,1}], \dots, [v_{1,m} \wedge \dots \wedge v_{k_m,m}]) &\mapsto  \\
[(v_{1,1} \wedge \dots \wedge v_{k_1,1})^{\otimes d_1} &\otimes  \dots \otimes (v_{1,m} \wedge \dots \wedge v_{k_m,m})^{\otimes d_m}].
\end{align*}

Note that  $\sigma\nu, \sigma\kappa, \sigma\psi, \nu\kappa, \nu\psi, \sigma\nu\kappa, \sigma\nu\psi$ are all compositions of their respective constituent maps. 
In fact we may use successive compositions of the embeddings $\sigma$ and $\nu$ to define more complicated algebraic varieties of the same nature; it is not possible to exhaust all such constructions. We refer the readers to \cite{GKZ94, Harris95:gtm133, Landsberg12} for basic properties of the Segre, Veronese, Grassmann, and Chow varieties from an algebraic geometric perspective, although none of which would be required for our subsequent discussions.

While we have defined all these varieties as projective varieties, it is important to note that for approximation problems, we will need to work in vector spaces rather than projective spaces. As such, when we discuss these varieties in the context of approximations, we would in fact be referring to their \emph{affine cones}: For any $X \subseteq \PP V$, its affine cone  is $\widehat{X} \coloneqq \pi^{-1}(X) \cup \{0\}$, where $\pi \colon V \setminus \{0\} \to \PP V$ is the quotient map taking a vector space $V$ onto its projective space $\PP V$. We will write $[v] \coloneqq \pi(v)$ for the projective equivalence class of $v \in V \setminus \{0\}$. 

\subsection{Best low-rank approximations of tensors}\label{sec:approx}

A variety $X \subseteq \PP V$ is said to be \emph{nondegenerate} if $X$ is not contained in any hyperplane. An implication is that its affine cone $\widehat{X}$ would span $V$, i.e., $\spn(\widehat{X}) = V$, and so any $p \in V$ can be expressed as a linear combination $p = \alpha_1 x_1 + \dots + \alpha_r x_r$ with $x_1,\dots,x_r \in \widehat{X}$. Since, by the definition of affine cone, $x \in \widehat{X}$ iff $\alpha x \in \widehat{X}$ for any $\alpha \ne 0$, we may in fact replace the linear combination by a sum, i.e., every $p \in V$ may be expressed as $p = x_1 + \dots + x_r$ for some $x_1,\dots,x_r \in \widehat{X}$.

Given an irreducible nondegenerate complex projective variety $X \subseteq \PP V$, the $X$-rank of a nonzero $ p \in V$ is defined as
\[
\rank_X (p) \coloneqq \min \{r \in \mathbb{N} : p = x_1 + \dots + x_r, \; x_i \in \widehat{X}\},
\]
and $\rank_X(0) \coloneqq 0$.  The \emph{$X$-border rank} of $p$, denoted by $\brank_X (p)$, is the minimum integer $r$ such that $p$ is a limit of a sequence of $X$-rank-$r$ points. For  irreducible nondegenerate complex projective varieties $X_1, \dots, X_r \subseteq \PP V$, the \emph{join map}  is defined by
\[
J: \widehat{X}_1 \times \cdots \times \widehat{X}_r \to V, \quad (x_1, \dots, x_r) \mapsto x_1 + \cdots + x_r.
\]
The Euclidean closure of the image $J(\widehat{X}_1 \times \cdots \times \widehat{X}_r)$ in $V$ is called the \emph{join variety} of $X_1, \dots, X_r$. In particular, when $X_1 = \cdots = X_r = X$, we denote the image of $J$  by $\Sigma_r^{\circ} (X)$, and its Euclidean closure by $\Sigma_r (X)$, often called the $r$-\emph{secant variety} of $X$.  Note that
\[
\Sigma^\circ_r(X) = \{ p \in V : \rank_X (p) \le r \}\qquad\text{and}\qquad
\Sigma_r(X) = \{ p \in V : \brank_X (p) \le r \}.
\]
A notion closely related to secant varieties is that of a \emph{tangent variety}, defined for a nonsingular projective variety, or more generally, for any projective variety as
\[
\tau (X) \coloneqq \bigcup_{x \in X} \widehat{\operatorname{T}}_x (X)\qquad
\text{or} \qquad
\tau (X) \coloneqq \bigcup_{x \in X\setminus X_{\sing}} \widehat{\operatorname{T}}_x (X)
\]
respectively. Here $\widehat{\operatorname{T}}_x (X)$ is the affine tangent space of $X$ at $x$ \cite[Section~8.1.1]{Landsberg12} and $X_{\sing}$ is the \emph{singular locus}, i.e., the subvariety of singular points in $X$, which has measure zero (in fact, positive codimension). When $X$ is nonsingular, $\tau (X)$ is an algebraic variety. When $X$ is singular, $\tau (X)$ is a quasiaffine variety. By abusing terminologies slightly, we call $\tau (X)$ a tangent variety in both cases. Among the varieties in Section~\ref{sec:tensor} that we are interested in, the Chow variety $\Ch_d(V)$ is singular when $\dim V > 2$ and thus so are the Segre--Chow, Veronese--Chow, Segre--Veronese--Chow varieties; the rest are all nonsingular varieties.

In fact, for each of the varieties listed in  Section~\ref{sec:tensor}, an element  $A \in \tau(X)$ has a  \emph{normal form} that is either well-known or straightforward to determine, i.e., we may write down an expression for an arbitrary  $A \in \tau(X)$.\footnote{In general this is not possible for higher order tensors, i.e., we do not have an analogue of Jordan normal form for \emph{all} $d$-tensors when $d > 2$; but in our case, we are restricting to a very small subset, namely, only the tensors in $\tau(X)$.}  
\begin{proposition}\label{prop:normal}
As each variety $X$ in  Section~\ref{sec:tensor} is defined by a map, we will denote a  tensor in the tangent space of that variety with a subscript given by the respective map. With this notation, the following is a list of normal forms for $\tau(X)$:
\begingroup
\allowdisplaybreaks
\begin{align*}
A_\sigma &= \sum_{i = 1}^d v_1 \otimes \cdots \otimes v_{i-1} \otimes w_i \otimes v_{i+1} \otimes \cdots \otimes v_d,\\ 
A_\nu &= v^{\circ (d-1)} \circ w,\\ 
A_\kappa &= \sum_{i = 1}^d v_1 \circ \cdots \circ v_{i-1} \circ w_i \circ v_{i+1} \circ \cdots \circ v_d,\\
A_\psi &= \sum_{i = 1}^k v_1 \wedge \cdots \wedge v_{i-1} \wedge w_i \wedge v_{i+1} \wedge \cdots \wedge v_k, \\
A_{\sigma\nu} &= \sum_{i = 1}^m v_1^{\otimes d_1} \otimes \cdots \otimes v^{\otimes d_{i-1}}_{i-1} \otimes (v^{\circ (d_i-1)}_i \circ w_i) \otimes v^{\otimes d_{i+1}}_{i+1} \otimes \cdots \otimes v^{\otimes d_m}_m,\\
A_{\sigma\kappa} &= \sum_{i=1}^m (v_{1,1} \circ \dots \circ v_{d_1,1})\otimes \cdots \otimes (v_{1,i-1} \circ \dots \circ v_{d_{i-1},i-1}) \\
&\qquad\otimes \biggl(\sum_{j = 1}^{d_i} v_{1, i} \circ \cdots \circ v_{j-1, i} \circ w_{j, i} \circ v_{j+1, i} \circ \cdots \circ v_{d_i, i} \biggr) \\
&\qquad\qquad\otimes (v_{1,i+1} \circ \dots \circ v_{d_{i+1},i+1}) \otimes \dots \otimes (v_{1,m} \circ \dots \circ v_{d_m,m}), \\
A_{\sigma\psi} &= \sum_{i=1}^m (v_{1,1} \wedge \dots \wedge v_{k_1,1})\otimes \cdots \otimes (v_{1,i-1} \wedge \dots \wedge v_{k_{i-1},i-1}) \\
&\qquad\otimes \biggl(\sum_{j = 1}^{k_i} v_{1, i} \wedge \cdots \wedge v_{j-1, i} \wedge w_{j, i} \wedge v_{j+1, i} \wedge \cdots \wedge v_{k_i, i}\biggr) \\
&\qquad\qquad\otimes (v_{1,i+1} \wedge \dots \wedge v_{k_{i+1},i+1}) \otimes \dots \otimes (v_{1,m} \wedge \dots \wedge v_{k_m,m}),\\
A_{\nu\kappa} &= (v_1 \circ \cdots \circ v_d)^{\circ (k-1)} \circ \biggl(\sum_{i = 1}^d v_1 \circ \cdots \circ v_{i-1} \circ w_i \circ v_{i+1} \circ \cdots \circ v_d\biggr),\\
A_{\nu\psi} &= (v_1 \wedge \cdots \wedge v_k)^{\circ (d-1)} \circ \biggl(\sum_{i = 1}^k v_1 \wedge \cdots \wedge v_{i-1} \wedge w_i \wedge v_{i+1} \wedge \cdots \wedge v_k\biggr),
\end{align*}
\begin{align*}
&A_{\sigma\nu\kappa} = \sum_{i=1}^m (v_{1,1} \circ \dots \circ v_{d_1,1})^{\otimes k_1} \otimes \cdots \otimes (v_{1,i-1} \circ \dots \circ v_{d_{i-1},i-1})^{\otimes k_{i-1}} \otimes \\
&\biggl[(v_{1,i} \circ \cdots \circ v_{d_i,i})^{\circ (k_i-1)} \circ \biggl(\sum_{j = 1}^{d_i} v_{1, i} \circ \cdots \circ v_{j-1,i} \circ w_{j,i} \circ v_{j+1,i} \circ \cdots \circ v_{d_i,i}\biggr)\biggr] \\
&\otimes (v_{1,i+1} \circ \dots \circ v_{d_{i+1},i+1})^{\otimes k_{i+1}} \otimes \dots \otimes (v_{1,m} \circ \dots \circ v_{d_m,m})^{\otimes k_m},\\
&A_{\sigma\nu\psi} = \sum_{i=1}^m (v_{1,1} \wedge \dots \wedge v_{k_1,1})^{\otimes d_1} \otimes \cdots \otimes (v_{1,i-1} \wedge \dots \wedge v_{k_{i-1},i-1})^{\otimes d_{i-1}} \otimes\\
&\biggl[(v_{1,i} \wedge \cdots \wedge v_{k_i,i})^{\circ (d_i-1)} \circ \biggl(\sum_{j = 1}^{k_i} v_{1, i} \wedge \cdots \wedge v_{j-1,i} \wedge w_{j,i} \wedge v_{j+1,i} \wedge \cdots \wedge v_{k_i,i}\biggr)\biggr] \\
&\qquad\otimes (v_{1,i+1} \wedge \dots \wedge v_{k_{i+1},i+1})^{\otimes d_{i+1}} \otimes \dots \otimes (v_{1,m} \wedge \dots \wedge v_{k_m,m})^{\otimes d_m}.
\end{align*}
\endgroup
\end{proposition}
\begin{proof}
Let $X$ be one of the varieties in Section~\ref{sec:tensor}. Then any $A \in \tau(X)$ is of the form
\[
A = \lim_{t \to 0} \frac{\gamma(t) - \gamma(0)}{t},
\]
where $\gamma$ is any complex analytic curve $\gamma(t) \subseteq \widehat{X} \setminus \{0\}$ such that $\gamma(t)$ is nonconstant around $0$ and $\gamma(0)$ is a nonsingular point in $\widehat{X}$. In fact, such a curve $\gamma(t)$ can be written down explicitly for any $X$ in Section~\ref{sec:tensor}. Let $\ast$ denote either $\circ$ or $\wedge$. Then any point in $\widehat{X} \setminus \{0\}$ takes the form
\[
(v_{1,1} \ast \cdots \ast v_{k_1, 1})^{\otimes d_1} \otimes \cdots \otimes (v_{1,m} \ast \cdots \ast v_{k_m, m})^{\otimes d_m},
\]
and any curve $\gamma(t) \subset \widehat{X} \setminus \{0\}$ is of the form 
\[
\gamma(t) = (v_{1,1} (t) \ast \cdots \ast v_{k_1, 1} (t))^{\otimes d_1} \otimes \cdots \otimes (v_{1,m} (t) \ast \cdots \ast v_{k_m, m} (t))^{\otimes d_m},
\]
where $v_{j, i} (t) \in V_j$ is a complex analytic curve with $v_{j, i} (0) = v_{j, i}$. Thus
\begin{align*}
&A = \lim_{t \to 0} \frac{\gamma(t) - \gamma(0)}{t} \\
&=\frac{d}{d t} (v_{1,1} (t) \ast \cdots \ast v_{k_1, 1} (t))^{\otimes d_1}\Bigr\rvert_{t=0}  \otimes (v_{1,2} \ast \cdots \ast v_{k_2, 2})^{\otimes d_2} \otimes \cdots \\
&\qquad \otimes  (v_{1,m} \ast \cdots \ast v_{k_m, m})^{\otimes d_m} + \cdots + (v_{1,1} \ast \cdots \ast v_{k_1, 1})^{\otimes d_1} \otimes \cdots  \\
&\qquad \otimes (v_{1,m-1} \ast \cdots \ast v_{k_{m-1}, m-1})^{\otimes d_{m-1}} \otimes \frac{d}{d t}  (v_{1,m} (t) \ast \cdots \ast v_{k_m, m} (t))^{\otimes d_m} \Bigr\rvert_{t=0} \\
&= \sum_{i = 1}^m (v_{1,1} \ast \dots \ast v_{k_1,1})^{\otimes d_1} \otimes \cdots \otimes (v_{1,i-1} \ast \dots \ast v_{k_{i-1},i-1})^{\otimes d_{i-1}}  \\
&\otimes \biggl[(v_{1,i} \ast \cdots \ast v_{k_i,i})^{\circ (d_i-1)} \circ \biggl(\sum_{j = 1}^{k_i} v_{1, i} \ast \cdots \ast v_{j-1,i} \ast w_{j,i} \ast v_{j+1,i} \ast \cdots \ast v_{k_i,i}\biggr)\biggr] \\
&\otimes (v_{1,i+1} \ast \dots \ast v_{k_{i+1},i+1})^{\otimes d_{i+1}} \otimes \dots \otimes (v_{1,m} \ast \dots \ast v_{k_m,m})^{\otimes d_m}
\end{align*}
for some $w_{j, i} \in V_j$. 
\end{proof}

The simplest examples with $\brank_X (A) \ne \rank_X (A)$ may be found in $\tau (X)$. Note that 
since every tangent is a limit of $2$-secants, we clearly have $\tau (X) \subseteq \Sigma_2 (X)$; but in general $\tau (X) \nsubseteq \Sigma_2^\circ (X)$. The next proposition gives sufficient conditions for a tensor $A$ in Proposition~\ref{prop:normal} so that $A \in \tau(X) \setminus \Sigma^\circ_2 (X)$, i.e.,
\[
\brank_X (A) = 2 < \rank_X (A).
\]
As each variety $X$ in  Section~\ref{sec:tensor} is defined by a map, we will denote the $X$-rank and border $X$-rank of a variety with a subscript given by the respective map.
\begin{proposition}\label{prop:rank}
Let $d,k,m \ge 3$. The tensors in $\tau(X)$ for any $X$ that is one of the varieties in Section~\ref{sec:tensor} have $X$-border-rank two and $X$-rank strictly greater than two given the respective sufficient conditions:
\begin{enumerate}[\upshape (i)]
\item If  $\{v_i,w_i\}$ is linearly independent for $i=1,\dots,d$, then 
\[\brank_\sigma (A_\sigma) = 2 < \rank_\sigma (A_\sigma).\]

\item If $\{v, w\}$ is linearly independent, then
\[\brank_\nu(A_\nu) = 2 < \rank_\nu(A_\nu).\]

\item If $[v_1], \dots, [v_d], [w_1], \dots, [w_d]$  are distinct in $\PP V$, then
\[\brank_\kappa(A_\kappa) = 2 < \rank_\kappa(A_\kappa).\]

\item If  $v_1 \wedge \cdots \wedge v_k \wedge w_1 \wedge \cdots \wedge w_k \ne 0$, then
\[\brank_\psi(A_\psi) = 2 < \rank_\psi(A_\psi).\]

\item If $[v_1], \dots, [v_m], [w_1], \dots, [w_m]$ are distinct in $\PP V$, then
\[\brank_{\sigma\nu} (A_{\sigma\nu}) = 2 < \rank_{\sigma\nu} (A_{\sigma\nu}).\]

\item If $[v_{1,1}], \dots, [v_{d_m,m}], [w_{1,1}], \dots, [w_{d_m,m}]$ are distinct in $\PP V$, then
\[\brank_{\sigma\kappa}(A_{\sigma\kappa}) =2 < \rank_{\sigma\kappa}(A_{\sigma\kappa}).\]

\item If $v_{1,i} \wedge \cdots \wedge v_{k_i,i} \wedge w_{1,i} \wedge \cdots \wedge w_{k_i,i} \ne 0$ for $i = 1, \dots, m$, then
\[\brank_{\sigma\psi}(A_{\sigma\psi}) = 2 < \rank_{\sigma\psi}(A_{\sigma\psi}).\]

\item If $[v_1], \dots, [v_d], [w_1], \dots, [w_d]$ are distinct in  $\PP V$, then
\[\brank_{\nu\kappa}(A_{\nu\kappa}) = 2 < \rank_{\nu\kappa}(A_{\nu\kappa}).\]

\item If  $v_1 \wedge \cdots \wedge v_k \wedge w_1 \wedge \cdots \wedge w_k \ne 0$, then
\[\brank_{\nu\psi} (A_{\nu\psi}) = 2 < \rank_{\nu\psi} (A_{\nu\psi}).\]

\item If $[v_{1,1}], \dots, [v_{d_m,m}], [w_{1,1}], \dots, [w_{d_m,m}]$ are distinct in  $\PP V$, then \[\brank_{\sigma\nu\kappa}(A_{\sigma\nu\kappa}) = 2 < \rank_{\sigma\nu\kappa}(A_{\sigma\nu\kappa}).\]

\item If $v_{1,i} \wedge \cdots \wedge v_{k_i,i} \wedge w_{1,i} \wedge \cdots \wedge w_{k_i,i} \ne 0$  for $i = 1, \dots, m$, then
\[\brank_{\sigma\nu\psi}(A_{\sigma\nu\psi} ) = 2 < \rank_{\sigma\nu\psi}(A_{\sigma\nu\psi} ).\]
\end{enumerate}
\end{proposition}
\begin{proof}
Since each of these tensors in Proposition~\ref{prop:normal} is a limit of the form
\[
A = \lim_{t \to 0} \frac{\gamma(t) - \gamma(0)}{t},
\]
we must have $\brank_X (A) \le 2$. 
On the other hand, because $\widehat{X}$ is complete as a metric space, $\rank_X (A) = 1$ if and only if $\brank_X (A) = 1$. Hence requiring that $\rank_X (A) > 1$ ensures that $\brank_X (A) > 1$. That the respective sufficient condition guarantees $\rank_X (A) > 2$ for each $X$ in Section~\ref{sec:tensor} follows from straightforward linear algebra. 
\end{proof}
Note that for each $X$ in  Section~\ref{sec:tensor}, the respective sufficient condition in Proposition~\ref{prop:rank} is satisfied by a general $A \in \tau(X)$. 

The \emph{generic $X$-rank}, denoted $r_{g}(X)$, is the minimum integer $r$ such that $\Sigma_r (X) = V$. Let $r, s\in \mathbb{N}$ be such that $1 < r < s \le r_g(X)$, and let $p \in V$ with $\rank_X(p) = s$. Then $p$ does not necessarily have a best $X$-rank-$r$ approximation. While such failures are well-known for Segre variety \cite{deSilvaLim08} and Veronese variety \cite{CGLM}, we deduce from Proposition~\ref{prop:rank} that they occur for all varieties  in Section~\ref{sec:tensor} when the orders of the tensors are higher than two.

\begin{theorem}\label{thm:nonexist}
Let  $X$ be any of the projective varieties listed in  Section~\ref{sec:tensor} and let $d,k,m \ge 3$. Then  there exist a tensor $A \in \spn(X)$, i.e., $A$ has the required symmetry and/or skew-symmetry, and an $r  < r_g(X)$ such that the best $X$-rank-$r$ approximation of $A$ does not exist, i.e.,
\[
\inf_{\rank_X(B) \le r} \lVert A - B \rVert
\]
is not attained by any tensor $B \in \spn(X)$ with $\rank_X(B) \le r$.
\end{theorem}
\begin{proof}
For each projective variety $X$ in Section~\ref{sec:tensor}, we need to exhibit a tensor $A \in \spn(X)$ with different $X$-rank and $X$-border rank; in which case $A$ will not have a best $X$-rank-$r$ approximation for $r = \brank_X(A)$. By Proposition~\ref{prop:rank}, we see that by setting $r = 2$ and picking a general point $A \in \tau(X)$ that satisfies the sufficient conditions in Proposition~\ref{prop:rank}, we obtain a tensor with no best $X$-rank-two approximation  for each of  the varieties $X$ in Section~\ref{sec:tensor}.
\end{proof}

While Theorem~\ref{thm:nonexist} shows that the phenomenon where a tensor fails to have a best $X$-rank-$r$ approximation can and does happen for every variety in Section~\ref{sec:tensor}, our next result is that such failures occur with probability zero, a consequence of  Theorem~\ref{thm:cmplxapp}. A high-level explanation goes as follows: When $r < r_g(X)$, the set of points in $\Sigma_r (X)$ whose $X$-rank is $r$ contains a Zariski open subset, and so the set of ``bad points'' in $\Sigma_r (X)$ --- those whose $X$-rank is not $r$ --- is contained in a subvariety; Theorem~\ref{thm:cmplxapp} then guarantees that these ``bad points'' are avoided  almost always.  A formal statement  follows next.
\begin{theorem}\label{thm:Xrkapprox}
Let $X \subseteq \PP V$ be an irreducible nondegenerate complex projective variety. Then a general point $p \in V$ has a unique best $X$-rank-$r$ approximation whenever $r < r_{g}(X)$.
\end{theorem}
When applied to  the varieties in Section~\ref{sec:tensor}, we obtain the following corollary.
\begin{corollary}
Let $V$ and $V_1,\dots,V_d$ be complex vector spaces.
\begin{enumerate}[\upshape (i)]
\item Any general  tensor $A \in V_1 \otimes \dots \otimes V_d$ has a unique best rank-$r$ approximation when $r < r_{g}\bigl(\sigma(\PP V_1 \times \dots \times \PP V_d)\bigr)$, i.e.,
\[
\inf_{v_{j,i} \in V_j} \biggl\lVert  A - \sum_{i = 1}^r v_{1,i} \otimes\dots \otimes v_{d,i} \biggr\rVert 
\]
can be attained.

\item Any general  symmetric tensor $A \in \mathsf{S}^d(V)$ has a unique best symmetric rank-$r$ approximation when $r < r_{g}\bigl(\nu (\PP V)\bigr)$, and a unique best Chow rank-$r$ approximation when $r < r_{g}\bigl(\Ch_d(V)\bigr)$, i.e.,
\[
\inf_{v_i \in V} \biggl\lVert  A - \sum_{i = 1}^r v^{\otimes d}_i \biggr\rVert \qquad \text{ and } \qquad \inf_{v_{j,i} \in V} \biggl\lVert A - \sum_{i = 1}^r v_{1,i} \circ \dots \circ v_{d,i} \biggr\rVert
\]
can be attained.

\item Any general  alternating $k$-tensor $A \in \mathsf{\Lambda}^k(V)$ has a unique best alternating rank-$r$ approximation when $r < r_{g}\bigl(\Gr_d(V)\bigr)$, i.e.,
\[
\inf_{v_{j, i} \in V} \biggl\lVert A - \sum_{i = 1}^r v_{1,i} \wedge \dots \wedge v_{k,i} \biggr\rVert
\]
can be attained.

\item Any general  $A \in \mathsf{S}^{d_1}(V) \otimes \dots \otimes \mathsf{S}^{d_m}(V)$ has a unique best Segre--Veronese rank-$r$ approximation when $r < r_{g}\bigl(\sigma\nu (\PP V_1 \times \dots \times \PP V_m)\bigr)$, and a unique best Segre--Chow rank-$r$ approximation when $r < r_{g}\bigl(\sigma\kappa (\Ch_{d_1}(V) \times \dots \times \Ch_{d_m}(V))\bigr)$, i.e.,
\begin{align*}
\inf_{v_{j, i} \in V} &\biggl\lVert A - \sum_{i = 1}^r v_{1,i}^{\otimes d_1} \otimes \dots \otimes v_{m,i}^{\otimes d_m} \biggr\rVert
\shortintertext{and}
\inf_{v_{j,l,i} \in V} &\biggl\lVert A - \sum_{i = 1}^r (v_{1,1,i} \circ \dots \circ v_{d_1,1,i}) \otimes \dots \otimes (v_{1,m,i} \circ \dots \circ v_{d_m,m,i}) \biggr\rVert
\end{align*}
can be attained.

\item Any general  tensor $A \in \mathsf{\Lambda}^{k_1}(V) \otimes \dots \otimes \mathsf{\Lambda}^{k_m}(V)$ has a unique best Segre--Grass\-mann rank-$r$ approximation when
\[
r < r_{g}\bigl(\sigma\psi (\Gr_{k_1}(V) \times \dots \times \Gr_{k_m}(V))\bigr), \text{ i.e.},
\]
\[
\inf_{v_{j, l, i} \in V} \biggl\lVert A - \sum_{i = 1}^r \bigl(v_{1,1,i} \wedge \dots \wedge v_{k_1,1,i}\bigr) \otimes \dots \bigl(v_{1,m,i} \wedge \dots \wedge v_{k_m,m,i}\bigr) \biggr\rVert
\]
can be attained.

\item Any general  tensor $A \in \mathsf{S}^k\bigl(\mathsf{S}^{d}(V)\bigr)$ has a unique best Veronese--Chow rank-$r$ approximation when $r < r_{g}\bigl(\nu\kappa (\Ch_d(V))\bigr)$, i.e.,
\[
\inf_{v_{j, i} \in V} \biggl\lVert A - \sum_{i = 1}^r (v_{1, i} \circ \dots \circ v_{d, i})^{\otimes k} \biggr\rVert
\]
can be attained.

\item Any general  tensor $A \in \mathsf{S}^d\bigl(\mathsf{\Lambda}^{k}(V)\bigr)$ has a unique best Veronese--Grassmann rank-$r$ approximation when $r < r_{g}\bigl(\nu\psi (\Gr_k(V))\bigr)$, i.e.,
\[
\inf_{v_{j, i} \in V} \biggl\lVert A - \sum_{i = 1}^r \bigl(v_{1, i} \wedge \dots \wedge v_{k, i}\bigr)^{\otimes d} \biggr\rVert
\]
can be attained.

\item Any general  tensor $A \in \mathsf{S}^{k_1}\bigl(\mathsf{S}^{d_1}(V)\bigr) \otimes \dots \otimes \mathsf{S}^{k_m}\bigl(\mathsf{S}^{d_m}(V)\bigr)$ has a unique best Segre--Veronese--Chow rank-$r$ approximation when $r < r_{g}\bigl(\sigma\nu\kappa(\Ch_{d_1}(V) \times \dots \times \Ch_{d_m}(V))\bigr)$, i.e.,
\[
\inf_{v_{j, l, i} \in V} \biggl\lVert A - \sum_{i = 1}^r (v_{1,1,i} \circ \dots \circ v_{d_1,1,i})^{\otimes k_1} \otimes \dots \otimes (v_{1,m,i} \circ \dots \circ v_{d_m,m,i})^{\otimes k_m} \biggr\rVert
\]
can be attained.

\item Any general  tensor $A \in \mathsf{S}^{d_1}\bigl(\mathsf{\Lambda}^{k_1}(V)\bigr) \otimes \dots \otimes \mathsf{S}^{d_m}\bigl(\mathsf{\Lambda}^{k_m}(V)\bigr)$ has a unique best Segre--Veronese--Grassmann rank-$r$ approximation when 
\[
r < r_{g}\bigl(\sigma\nu\psi (\Gr_{k_1}(V) \times \dots \times \Gr_{k_m}(V))\bigr), \text{ i.e.},
\]
\[
\inf_{v_{j, l,i} \in V} \biggl\lVert A - \sum_{i = 1}^r (v_{1,1,i} \wedge \dots \wedge v_{k_1,1,i})^{\otimes d_1} \otimes \dots (v_{1,m,i} \wedge \dots \wedge v_{k_m,m,i})^{\otimes d_m} \biggr\rVert
\]
can be attained.
\end{enumerate}
\end{corollary}

\section{Applications to other approximation problems}\label{sec:join}

Our general existence and uniqueness result in Theorem~\ref{thm:cmplxapp} has implications beyond approximation by secant varieties in tensor spaces. In this section, we will look at some other approximation problems that arise in practical applications but do not fall naturally into the class of approximation problems considered in Section~\ref{sec:approx}

\subsection{\blue{Sparse-plus-low-rank approximations}}\label{sec:splr}

Discussions of sparsity necessarily involve a choice of bases and so in this section we will assume that bases have been chosen on our vector spaces and that $V_1 = \mathbb{C}^{n_1}, \dots, V_d = \mathbb{C}^{n_d}$.  We start by considering the linear subspace of $k$-sparse hypermatrices with a \emph{fixed} sparsity pattern in $\mathbb{C}^{n_1} \otimes \dots \otimes \mathbb{C}^{n_d} \cong \mathbb{C}^{n_1 \times \dots \times n_d}$ and leave the case where the sparsity pattern is not fixed to later. More precisely, the sparsity pattern, $\Theta \subseteq \{1,\dots,n_1\} \times \dots \times \{1,\dots,n_d\}$, is a subset of the index set and
\[
S_\Theta \coloneqq \{C \in \mathbb{C}^{n_1 \times \dots \times n_d} \mid c_{i_1\cdots i_d} = 0 \text{ if } (i_1,\dots,i_d) \notin \Theta \}.
\]
In the language of functions, this is just the vector space of complex-valued functions supported on $\Omega$.

Denote the set of hypermatrices of rank not more than $r$ by
\begin{equation}\label{eq:Rr}
R_r \coloneqq   \{B \in \mathbb{C}^{n_1 \times \dots \times n_d} \mid \rank(B) \le r \},
\end{equation}
and recall that its closure projectivizes to the $r$-secant variety of the Segre variety
\[
\PP \overline{R}_r = \Sigma_r (\sigma(\PP \mathbb{C}^{n_1} \times \dots \times \PP \mathbb{C}^{n_d})).
\]
Given a hypermatrix $A \in \mathbb{C}^{n_1 \times \dots \times n_d}$, a best $k$-sparse-plus-rank-$r$  approximation of $A$ is a solution to the problem:
\begin{equation}\label{sparseapp}
\inf_{B \in R_r, \; C \in S_\Theta} \|A - (B + C)\|.
\end{equation}
In other words, we would like study the best approximation problem whereby $A$ is approximated by a hypermatrix in $J\bigl(\PP \overline{R}_r, \PP S\bigr)$. Recall that this is a join variety, as defined in Section~\ref{sec:approx}. By Theorem~\ref{thm:cmplxapp}, we may deduce the following existence and uniqueness result for such approximation problems.
\begin{corollary}\label{cor:splr}
For a general complex hypermatrix $A \in \mathbb{C}^{n_1 \times \dots \times n_d}$, the problem \eqref{sparseapp} has a unique solution $B + C$ such that $\rank(B) = r$ and $C \in S_\Theta$ is a $k$-sparse matrix with sparsity pattern $\Theta$.
\end{corollary}

Such scenarios where one requires a fixed sparsity arise in hypermatrix completion problems like the one in Section~\ref{sec:compl} and in statistical models like the one in Section~\ref{sec:factor}. However, the better known  $k$-sparse-plus-rank-$r$  approximation problem as stated in \cite{CLMW11:jacm, CSPW11:siamo} does not impose a fixed sparsity pattern on $C$ but allows it be any $k$-sparse matrix. In other words, we consider the set of $k$-sparse matrices
\[
S_k \coloneqq \bigcup_{\# \Theta = k} S_\Theta =\{C \in \mathbb{C}^{n_1 \times \dots \times n_d} \mid \lVert C \rVert_0 \le k \},
\]
where $\lVert C \rVert_0 $ counts the number of nonzero entries in $C$.
Note that such a union will no longer be a subspace of $\mathbb{C}^{n_1 \times \dots \times n_d}$. In this case, we consider the \emph{reducible} variety
\[
\bigcup_{\#\Theta =k } J\bigl(\Sigma_r (\sigma(\PP \overline{R}_r), \PP S_\Theta\bigr)
\]
and apply Corollary~\ref{cor:imofpoly} to  obtain our required existence and uniqueness result.
\begin{corollary}\label{cor:splr2}
For a general complex hypermatrix $A \in \mathbb{C}^{n_1 \times \dots \times n_d}$, the problem:
\[
\inf_{\rank(B) \le r, \; \lVert C \rVert_0 \le k } \|A - (B + C)\|
\]
has a unique solution $B + C$ such that $\rank(B) = r$ and $\lVert C\rVert_0 = k$.
\end{corollary}
Note that in both Corollaries~\ref{cor:splr} and \ref{cor:splr2}, the uniqueness applies to the solution $B +C$, i.e., not to $B$ and $C$ individually but to their sum.

\subsection{\blue{Hypermatrix completion}}\label{sec:compl}

The discussions in this section again depend on a choice of bases and we make the same assumptions as in 
Section~\ref{sec:splr}. Suppose we are given the values of the entries
\[
\{  a_{i_1\cdots i_d} \mid (i_1,\dots,i_d) \in \Theta \}
\]
of a hypermatrix $A \in \mathbb{C}^{n_1 \times \dots \times n_d}$ with $\Theta \subseteq \{1,\dots,n_1\} \times \dots \times \{1,\dots,n_d\}$ a subset of the indices. Let
\[
P_\Theta : \mathbb{C}^{n_1 \times \dots \times n_d} \to \mathbb{C}^{n_1 \times \dots \times n_d}, \quad 
A \mapsto \sum_{(i_1,\dots,i_d) \in \Theta} a_{i_1\cdots i_d} e_{i_1} \otimes \dots \otimes e_{i_d},
\]
be the projection onto the subspace defined by  coordinates of the known entries.
The problem of finding $B  \in \mathbb{C}^{n_1 \times \dots \times n_d}$ of rank at most $r$ that attains
\begin{equation}\label{eq:compl}
\inf_{\rank(B)\le r} \|P_\Theta(A - B)\|
\end{equation}
is called a best rank-$r$ \emph{hypermatrix completion problem}.\footnote{Often called ``tensor completion'' but this is a misnomer since the problem clearly depends on coordinates.} In this situation our result implies the following. 
\begin{corollary}\label{cor:matcomp}
Let $\Theta \subseteq \{1,\dots,n_1\} \times \dots \times \{1,\dots,n_d\}$.
If the given entries $\{  a_{i_1\cdots i_d} \mid (i_1,\dots,i_d) \in \Theta \}$ of a hypermatrix $A \in \mathbb{C}^{n_1 \times \dots \times n_d}$ are general, then there exists a hypermatrix $B\in \mathbb{C}^{n_1 \times \dots \times n_d}$ of rank $r$ that solves \eqref{eq:compl}. Furthermore, although $B$ may be not unique, its entries $\{  b_{i_1\cdots i_d} \mid (i_1,\dots,i_d) \in \Theta \}$ are unique.
\end{corollary}
\begin{proof}
Consider the image of rank-$r$ hypermatrices  $P_\Theta(R_r)$ with $R_r$ as in \eqref{eq:Rr}.  As $P_\Theta(A-B)=P_\Theta(A)-P_\Theta(B)$, a best rank-$r$ hypermatrix completion problem is equivalent to approximating $P_\Theta(A)$ by an element of $P_\Theta(R_r)$. Applying Theorem~\ref{thm:cmplxapp} with $X$ as the closure of $P_\Theta(R_r)$ and $Z = X \setminus P_\Theta(R_r)$, we obtain a unique solution $P_\Theta(B) \in P_\Theta(R_r)$, i.e.,  there exists $B$ of rank $r$ that is a best rank-$r$ completion of $A$. Note that $B$ is not unique but two different $B$'s must have the same projection $P_\Theta(B)$, i.e., the entries $\{  b_{i_1\cdots i_d} \mid (i_1,\dots,i_d) \in \Theta \}$ are unique.
\end{proof}

The similarity of the best rank-$r$ hypermatrix completion problem  \eqref{eq:compl}  and the best $k$-sparse-plus-rank-$r$  approximation problem \eqref{sparseapp} in Section~\ref{sec:splr} should not go unnoticed. For a given index set $\Theta$, the solution $B$ to \eqref{eq:compl} with $\Theta$ as the indices of known entries equals the $B$ in  \eqref{sparseapp}  with $\Theta$ as the sparsity pattern and $C \in \ker(P_\Theta)$.

\subsection{Gaussian $r$-factor analysis model with $k$ observed variables}\label{sec:factor}

Consider a Gaussian hidden variable model with $k$ observed variables $X_1,\dots,X_k$ and $r$ hidden variables $Y_1,\dots,Y_r$ where $(X_1, \dots, X_k,  Y_1, \dots, Y_r)$ follows a joint multivariate normal distribution with positive definite covariance matrix. If the observed variables are conditionally independent given the hidden variables, this model is called the \emph{Gaussian $r$-factor analysis model with $k$ observed variables} \cite{DrtonSturmfelsSullivant07}  and denoted by $\mathbf{F}_{k, r}$. In fact, by \cite[Proposition~1]{DrtonSturmfelsSullivant07}, $\mathbf{F}_{k, r}$ is the family of multivariate normal distributions $\mathcal{N} (\mu, \Sigma)$ on $\RR^k$ with $\mu \in \RR^k$ and $\Sigma$ belonging to 
\[
F_{k, r} \coloneqq \{ \Psi + LL^\tp  \in \RR^{k \times k} \mid \Psi \succ 0 \; \text{and diagonal,} \; L \in \RR^{k \times r}\}.
\]
The standard approach in \emph{algebraic statistics} \cite{DSS} and also that in \cite{DrtonSturmfelsSullivant07}  is to drop any semialgebraic conditions and complexify. In this case, it means to drop the condition $\Sigma \succ 0$ and regard all quantities as complex valued, i.e.,
\[
F_{k, r}(\mathbb{C}) = \{ \Psi + LL^\tp  \in \CC^{k \times k} \mid \Psi \; \text{diagonal,}\;  L \in \CC^{k \times r}\}.
\]
This is the set of complexified covariance matrices for the model $\mathbf{F}_{k, r}$ and  the algebraic approach undertaken in \cite{DrtonSturmfelsSullivant07} effectively treats $F_{k, r}(\mathbb{C})$ as the parameter space of the Gaussian $r$-factor analysis model with $k$ observed variables.

If we replace the space of matrices $V \otimes W$  in Section~\ref{sec:splr} by the space of symmetric matrices $\mathsf{S}^2(V)$ and let $S \subseteq \mathsf{S}^2(V)$ be the subspace of diagonal matrices, then we see that  
\[
F_{k, r}(\mathbb{C}) = J\bigl(\Sigma_r (\nu_2(\PP V)), \PP S\bigr).
\]
It follows from Corollary~\ref{cor:splr} that every general complexified covariance matrix $\Sigma \in \mathsf{S}^2(V)$  has a unique best approximation by  $\Psi + LL^\tp$, a complexified covariance matrix for the model $\mathbf{F}_{k, r}$. To provide context, readers unfamiliar with factor analysis \cite[Chapter~9]{JW} should note that its main parameter estimation problem is to determine the matrix of \emph{loadings} $L$ and the diagonal matrix of \emph{specific variances} $\Psi = \diag(\psi_1,\dots,\psi_k)$ from a sample covariance matrix $\Sigma$.

\subsection{Block-term tensor approximations}

Our final example of finding a best approximation in a join variety brings us back to tensors. A  class of tensor approximation problems with groundbreaking applications in signal processing is the so-called \emph{block-term decompositions} \cite{DeLathauwer08}. Unlike the factor analysis application in Section~\ref{sec:factor}, which is really a problem over $\mathbb{R}$ but is complexified to allow techniques of algebraic statistics, the use of block-term decompositions as a model in signal processing is naturally and necessarily over $\mathbb{C}$.

Any tensor $A \in V_1 \otimes \dots \otimes V_d$  induces a linear map $A_i \colon V_i^* \to V_1 \otimes \dots \otimes V_{i-1} \otimes V_{i+1} \otimes \dots \otimes V_d$. The rank of this linear map is of course just  the dimension of its image $\dim A_i(V^*_i)$. The \emph{multilinear rank} of $A$ is then defined to be the $d$-tuple
\begin{align*}
\mrank(A) &\coloneqq \bigl(\dim A_1(V^*_1), \dots, \dim A_d(V^*_d)\bigr)
\intertext{and the set}
\Sub_{r_1, \dots, r_d} (V_1 \otimes \dots \otimes V_d) &\coloneqq \{A \in V_1 \otimes \dots \otimes V_d \mid \mrank(A) \le (r_1,\dots,r_d)\}
\end{align*}
is a complex algebraic variety called a \emph{subspace variety}. 

Given  $(r_{1,1},\dots,r_{1,d}), \dots, (r_{k,1},\dots,r_{k,d})$, a tensor $B \in V_1 \otimes \dots \otimes V_d$ is said to have a block-term decomposition if
\[
B = B_1 + \dots + B_k,\qquad \mrank(B_i) \le (r_{i, 1}, \dots, r_{i, d}), \quad i =1,\dots,k.
\]
It is known \cite{dMGC, DeLathauwer08} that the best block-term approximation problem
\begin{equation}\label{eq:block}
\inf \bigl\{ \lVert A - (B_1 + \dots + B_k) \rVert  \bigm|  \mrank(B_i) \le (r_{i, 1}, \dots, r_{i, d}), \; i =1, \dots, k \bigr\}
\end{equation}
does not have a solution for certain choices of $A$, i.e., the infimum above cannot be attained, much like the tensor approximation problems we discussed in Section~\ref{sec:approx}; \blue{an explicit example where the infimum in \eqref{eq:block} is unattainable for a set of tensors with nonempty interior can be found in \cite{dMGC}.}

Theorem~\ref{thm:cmplxapp}, when applied to the join variety of $k$ subspace varieties, gives us the following.
\begin{corollary}
Let $V_1,\dots,V_d$ be complex vector spaces.
A general tensor $A \in V_1 \otimes \dots \otimes V_d$ has a unique best approximation in the join variety
\[
J\bigl(\Sub_{r_{1, 1}, \dots, r_{1, d}} (V_1 \otimes \dots \otimes V_d), \dots, \Sub_{r_{k, 1}, \dots, r_{k, d}} (V_1 \otimes \dots \otimes V_d)\bigr),
\]
i.e., the best block-term approximation problem in \eqref{eq:block} has a unique solution.
\end{corollary}

\subsection{\blue{Approximations by tensor networks}}

Tensor networks are tensors or functions that have separable decompositions indexed by an undirected graph.  For example a \emph{matrix product state} $f \in L^2(\Omega_1 \times \Omega_2 \times \Omega_3)$ corresponding to a triangle takes the form
\[
f = \sum_{i,j,k=1}^{p,q,r} \varphi_{ij} \otimes \psi_{jk} \otimes \theta_{ki}
\]
where $\varphi_{ij} \in L^2(\Omega_1)$, $\psi_{jk} \in L^2(\Omega_2)$, $\theta_{ki} \in L^2(\Omega_3)$; note that the indices $i,j,k$ form a triangle with edges $\{i,j\}$, $\{j,k\}$, $\{k,i\}$ (see \cite{ye2018tensor} for more details). Other undirected graphs give other tensor networks, examples include: matrix product states (cycle graphs), tensor trains (path graphs), tree tensor network states (trees),  projected entangled pair states (product of path graphs), etc. For a precise definition, see \cite[Definition 2.1]{ye2018tensor}.

Tensor networks play a prominent role in quantum physics and quantum chemistry \cite{hackbusch2012tensor, perez2007matrix, verstraete2008matrix, schollwock2011density, orus2014practical, szalay2015tensor}. 
It is known that tensor networks $X_G$ corresponding to cyclic graphs $G$ are not closed \cite{harris2018computing, landsburg2012geometry} and thus may not have best $X_G$-approximation. 

\begin{corollary}\label{cor:net}
Let $G$ be any undirected graph with $d$ vertices and $V_1,\dots,V_d$ be complex vector spaces. If $X_G \subseteq V_1\otimes\dots\otimes V_d$ is a set of tensor network states corresponding to $G$. Then the best $\overline{X}_G$-approximation of a general tensor is unique and lies in $X$. 
\end{corollary}
\begin{proof}
By Theorem~\ref{thm:cmplxapp} it is enough to show that $X_G$ contains a Zariski open subset of $\overline{X}_G$. For tensor networks this follows from \cite[Proposition 2.6]{ye2018tensor}.
\end{proof}
Corollary~\ref{cor:net} may be extended to other variants of tensor networks such as uniform (or site-independent) matrix product states, which are symmetric versions of matrix product states \cite{verstraete2008matrix}; the analogue of Corollary~\ref{cor:net} in this case follows from \cite[Proposition 4.2]{harris2018computing}.

\section{Conclusion}

Our study in this article demonstrates a significant difference between approximation problems over $\mathbb{R}$ and over $\mathbb{C}$. In the case when the approximation involves tensors, this revelation is  consistent with our knowledge of other aspects of tensors --- tensor rank, tensor spectral norm, tensor nuclear norm are all known to be dependent on the base field \cite{deSilvaLim08, nuclear}; even the fact that the Grothendieck constants have different values over $\mathbb{R}$ and over $\mathbb{C}$ is a manifestation of this phenomenon \cite{ZL}.

The knowledge that over $\mathbb{C}$, the best rank-$r$ approximation problems for tensors and, more generally, any best $r$-term approximation problems, are  well-posed except on a set of measure zero should provide a reasonable amount of justification for applications that depend on such approximations. In fact, such ``almost everywhere''  guarantees are about as much as one may hope for, since ``everywhere'' is certainly false --- as we saw, for every single approximation problem that we considered in this article, there are indeed instances where best approximations do not exist.

\section*{Acknowledgment}

We are grateful to J.~M.~Landsberg for posing the question about general existence of best low-rank approximations of tensors over $\mathbb{C}$ to us (when LHL visited College Station in 2010, and when he visited YQ in Grenoble in 2015). MM thanks W.~Hackbusch for inspiring discussions on the topic. Special thanks go to the anonymous referees for their careful reading and numerous invaluable suggestions that significantly enriched the content of this article.

The work of YQ and LHL is supported by DARPA D15AP00109 and NSF IIS 1546413. The work of MM is supported by the Polish National Science Center grant 2013/08/A/ST1/00804. In addition LHL acknowledges support from a DARPA Director's Fellowship and the Eckhardt Faculty Fund.

\end{document}